\documentclass[12pt, reqno, a4paper]{amsart}

\usepackage{amssymb}
\usepackage{color}
\definecolor{darkgreen}{rgb}{0,0.45,0} 
\usepackage[
colorlinks,citecolor=blue,linkcolor=blue, breaklinks]{hyperref}
\usepackage[matrix, arrow, curve]{xy}
\usepackage{bbm}
\usepackage{xcolor}
\usepackage{enumitem}

\usepackage[english]{babel}

\binoppenalty=\maxdimen
\relpenalty=\maxdimen

\theoremstyle{plain}
\newtheorem{theorem}{Theorem}[section]
\newtheorem{lemma}[theorem]{Lemma}
\newtheorem{proposition}[theorem]{Proposition}
\newtheorem{corollary}[theorem]{Corollary}

\theoremstyle{remark}
\newtheorem{remark}[theorem]{Remark}

\theoremstyle{definition}
\newtheorem{example}[theorem]{Example}

\newtheorem{definition}[theorem]{Definition}

\numberwithin{equation}{section}

\DeclareMathOperator{\id}{id}
\DeclareMathOperator{\ch}{char}
\DeclareMathOperator{\Ker}{Ker}

\DeclareMathOperator{\GL}{GL}

\DeclareMathOperator{\End}{End}
\DeclareMathOperator{\Aut}{Aut}

\DeclareMathOperator{\supp}{supp}
\DeclareMathOperator{\cosupp}{cosupp}

\DeclareMathOperator{\Hom}{Hom}

\usepackage{graphicx}
\usepackage[geometry]{ifsym}

\renewcommand{\square}{\text{\normalfont\scalebox{.6}{\SmallSquare}}}

\DeclareRobustCommand{\No}{\ifmmode{\nfss@text{\textnumero}}\else\textnumero\fi} 

\newbox\skewpullbackbox
\setbox\skewpullbackbox=\hbox{\xy 0;<1mm,0mm>: \POS(4,0)\ar@{-} (-4,0) \ar@{-} (8,4)
	\endxy}

\newbox\skwepullbackbox
\setbox\skwepullbackbox=\hbox{\xy 0;<1mm,0mm>: \POS(16,0)\ar@{-} (10,0) \ar@{-} (12,4)
	\endxy}

\newbox\ksewpullbackbox
\setbox\ksewpullbackbox=\hbox{\xy 0;<1mm,0mm>: \POS(0,-8)\ar@{-} (0,-4) \ar@{-} (4,-4)
	\endxy}

\newbox\pullbackbox
\setbox\pullbackbox=\hbox{\xy 0;<1mm,0mm>: \POS(4,0)\ar@{-} (0,0) \ar@{-} (4,4)
	\endxy}

\newbox\pullbackabox
\setbox\pullbackabox=\hbox{\xy 0;<1mm,0mm>: \POS(-4,-6)\ar@{-} (-8,-6) \ar@{-} (-4,-2)
	\endxy}

\newbox\pullbackbbox
\setbox\pullbackbbox=\hbox{\xy 0;<1mm,0mm>: \POS(-4,-4)\ar@{-} (-8,-4) \ar@{-} (-4,0)
	\endxy}

\newbox\pullbackcbox
\setbox\pullbackcbox=\hbox{\xy 0;<1mm,0mm>: \POS(4,-7)\ar@{-} (0,-7) \ar@{-} (4,-3)
	\endxy}

\newbox\pullbackdbox
\setbox\pullbackdbox=\hbox{\xy 0;<1mm,0mm>:\POS(-10,-6)\ar@{-} (-14,-6) \ar@{-} (-10,-2)
	\endxy}

\newbox\pushoutbox
\setbox\pushoutbox=\hbox{\xy 0;<1mm,0mm>: \POS(0,4)\ar@{-} (0,0) \ar@{-} (4,4)
	\endxy}

\newbox\pushoutabox
\setbox\pushoutabox=\hbox{\xy 0;<1mm,0mm>: \POS(2,8)\ar@{-} (2,4) \ar@{-} (8,8)
	\endxy}

\newcounter{dummy}
\makeatletter
\newcommand\myitem[1][]{\item[(#1)]\refstepcounter{dummy}\def\@currentlabel{#1}}
\makeatother

\makeatletter
\@namedef{subjclassname@2020}{\textup{2020} Mathematics Subject Classification}
\makeatother

\begin{document}

\title{On the classification of quantum symmetries}

\author{A.\,S. Gordienko}

\email{alexey.gordienko@math.msu.ru}

\author{A.\,I. Pekarsky}

\email{aleksandr.pekarskii@math.msu.ru}

\address{Department of Higher Algebra,
	Faculty of Mechanics and  Mathematics,
	M.\,V.~Lomonosov Moscow State University,
	Leninskiye Gory, d.1,  Main Building, GSP-1, 119991 Moskva, Russia}


\begin{abstract} We show that, in order to classify Hopf algebra (co)actions on a given finite dimensional algebra up to equivalence, one should start with the classification of the possible \textit{cosupports} (i.e. the sets of linear operators by which $H^*$ is acting) of Hopf algebra coactions
	and then consider dual Hopf algebra actions. As an application, we classify quantum symmetries of the set of two points and the algebra of dual numbers. In addition, we show that the straight line does not admit an (ungraded) universal coacting Manin Hopf algebra.
	Moreover, we prove that for $n\geqslant 14$  the full matrix algebra $M_n(\mathbbm k)$
	admits a nontrivial Hopf algebra coaction such that all Hopf algebra actions with the same restriction on the cosupport are trivial, i.e. the cosupport may reduce under the dualization and a finite dimensional algebra may have less equivalence classes of actions than coactions.
	Simultaneously, for any $n\geqslant 14$, we define an elementary grading  on $M_n(\mathbbm k)$ by an infinite group
	that cannot be regraded by any finite group. (The previously known lower bound was $n\geqslant 349$.)
\end{abstract}


\thanks{The authors were supported by the Theoretical Physics and Mathematics Advancement Foundation ``BASIS''.
	The study was conducted under the state assignment of Lomonosov Moscow State University.}

\maketitle

\section{Introduction}

When an affine algebraic group $G$ is acting morphically on an affine algebraic variety $X$,
the algebra $\mathcal O (X)$ of regular functions on $X$ is a comodule algebra over the Hopf algebra $\mathcal O(G)$
and the module algebra over both the group Hopf algebra $\mathbbm k G$ and the universal enveloping algebra $U(\mathfrak g)$
of the Lie algebra of $G$ where $\mathbbm k$ is a base field. Note that in this situation, where one can say that $G$ is acting on $X$ by \textit{classical symmetries},
$\mathcal O(X)$ and $\mathcal O(G)$ are commutative and $\mathbbm k G$ and $U(\mathfrak g)$ are cocommutative.
When a not necessarily (co)commutative Hopf algebra $H$ is (co)acting on a not necessarily commutative algebra $A$, one can
interpret this as that the quantum group $H$ is acting by \textit{quantum symmetries} on the noncommutative space whose ``algebra of functions'' is $A$. 
This very general approach traces its origins back to Yu.\,I.~Manin's notes on quantum groups and noncommutative geometry~\cite{Manin}.
	On the other hand, the problem of the classification of quantum symmetries of a given algebra can be viewed as a generalization of the problem of the classification of automorphisms, group gradings, (skew)derivations, etc.
	While the Manin Hopf algebra is universal among all not necessarily commutative coacting Hopf algebras,
	one can consider the universal commutative coacting Hopf algebra, which is precisely the Hopf algebra corresponding to the automorphism group scheme~\cite[Sections 1.5 and 7.6]{Waterhouse}. Hence the Manin Hopf algebra can be viewed as the noncommutative counterpart of the automoprhism group scheme.	
	Some recent work on Hopf actions
	and their properties includes, but is not limited to the papers~\cite{AK2024, BM21, CY20, CKWZ16, CWZ23, CG20, CE17, CEW16, EN22, EW16, EW17, GW22, KO21, LNY20}.

In some cases group gradings are classified up to isomorphism and in some cases up to equivalence, see e.g.~\cite{ElduqueKochetov}.
While the classification up to equivalence may seem coarser, the classification up to isomorphism can be essentially recovered
using the notion of the universal group of the grading
introduced by J.~Patera and H.~Zassenhaus in 1989~\cite{PZ89}. The notion of equivalence of Hopf algebra (co)actions
as well as their universal (co)acting Hopf algebras were introduced in~\cite{ASGordienko20ALAgoreJVercruysse}.
The latter were united with the Sweedler~--- Manin~--- Tambara universal (co)acting bi/Hopf algebras into a single theory
in~\cite{ASGordienko21ALAgoreJVercruysse} via the notion of a $V$-universal (co)acting Hopf algebra.  Moreover, the Duality Theorem~\cite[Theorem~4.15]{ASGordienko21ALAgoreJVercruysse} for $V$-universal (co)acting Hopf algebra was proved. A purely categorical approach to quantum symmetries was undertaken in~\cite{ASGordienko24ALAgoreJVercruysse, ASGordienko25ALAgoreJVercruysse}.

While there are several general results on universal (co)acting Hopf algebras and duality between them, 
the only case where all Hopf algebra actions have been completely classified so far was the case of the algebra $\mathbbm k[x]/(x^2)$ of dual numbers~\cite[Theorem~6.7]{ASGordienko20ALAgoreJVercruysse}, however even in this case
the universal acting Hopf algebra was calculated only for two of three possible equivalence classes of actions.
In Section~\ref{SectionClassificationOutline} we propose a general classification strategy for quantum symmetries of finite dimensional algebras, which we follow in Sections~\ref{SectionTwoPoints} and~\ref{SectionDualNumbers}
and classify all quantum symmetries (both Hopf algebra actions and coactions) of the set of two points and of $\mathbbm k[x]/(x^2)$. Furthermore, we calculate all the corresponding universal (co)acting Hopf algebras.

The Duality Theorem implies that for every equivalence class of Hopf algebra actions on a finite dimensional algebra $A$
there exists an equivalence class of Hopf algebra coactions on $A$ with the same cosupport. Therefore, there cannot be more
equivalence class of actions than coactions. The natural question arises as to whether a bijection always exists between the sets
of equivalence classes of actions and coactions. This question is equivalent to the following: can the cosupport of the $V$-universal Hopf algebra action dual to a $V$-universal Hopf algebra coaction $\rho^{\mathbf{Hopf}}_{A,V}$
with $\cosupp \rho^{\mathbf{Hopf}}_{A,V} = V$
become smaller than $V$? In Section~\ref{SectionCosuppReduces} we show that the answer is ``yes''.
In fact, for any $n\geqslant 14$, we define an elementary grading on the full matrix algebra $M_n(\mathbbm k)$ by an infinite group
that cannot be regraded by any finite group and show that the action dual to the corresponding group algebra coaction is trivial.
Simultaneously, we solve another problem since the previously known lower bound
for $n$ such that on $M_n(\mathbbm k)$ there exists an elementary grading that cannot be regraded by a finite group was $n\geqslant 349$, see~\cite{ASGordienko18Schnabel}.

While the universal acting Hopf algebras always exist~\cite[Section~4.1]{ASGordienko21ALAgoreJVercruysse},
the universal coacting Hopf algebras do not necessarily exist for infinite dimensional algebras. The first examples were constructed
in~\cite[Example~4.3]{ASGordienko21ALAgoreJVercruysse}. In Theorem~\ref{TheoremYuIManinHopfAlgebraStraightLineDoesNotExist}
we show that the (ungraded) Manin universal (co)acting Hopf algebra does not exist even for a straight line.

\section{Group gradings}\label{SectionGroupGradings}

We begin with group gradings, as they are an important particular case of quantum symmetries (see the precise definition of the latter in Section~\ref{Subsection(Co)moduleAlgebras} below) and many definitions and constructions related to quantum symmetries are inspired by those of gradings.

Recall that $\Gamma \colon A=\bigoplus_{g \in G} A^{(g)}$ (direct sum of subspaces)
is a \textit{grading} on a (not necessarily associative) algebra $A$ over a field $\mathbbm k$ by a group $G$ if $A^{(g)}A^{(h)}\subseteq A^{(gh)}$ for all $g,h\in G$.
Then $G$ is called the \textit{grading group} of $\Gamma$.
The algebra $A$ is called \textit{graded} by $G$.

\begin{example} Let $\mathbbm k$ be a field and let $g_k \in G$ for $1\leqslant k \leqslant n$ for some group $G$
	and $n\in\mathbb N$.  An \textit{elementary grading} on the full matrix algebra $M_n(\mathbbm k)$
	defined by the $n$-tuple $(g_1, \dots, g_n)$
	is the grading where for every $1\leqslant i,j\leqslant n$ the matrix unit $e_{ij}$ belongs to $M_n(\mathbbm k)^{\left(g_i g_j^{-1}\right)}$.
\end{example}
\begin{remark}\label{RemarkElementary}
	Note that such a grading is uniquely determined by defining the $G$-degrees of $e_{i,i+1}$, $1\leqslant i \leqslant n-1$. If $G$ is an arbitrary group and $(h_1, \dots, h_{n-1})$ is an arbitrary $(n-1)$-tuple of elements of $G$, then the elementary grading with $e_{i,i+1} \in M_n(\mathbbm k)^{(h_i)}$ can be defined by $(g_1, \dots, g_n)$ where $g_i = \prod\limits_{j=i}^{n-1} h_j$. If a graded algebra $A$ is unital,
	then its identity element~$1_A$ belongs to $A^{(1)}$. Therefore, elementary gradings on full matrix algebras are precisely the gradings where the matrix units are \textit{homogeneous elements}, i.e. belong to the union of the graded components $M_n(\mathbbm k)^{(g)}$.
\end{remark}

When studying graded algebras, one has to determine, when two
graded algebras can be considered ``the same'' or equivalent.

Let \begin{equation}\label{EqTwoGroupGradings}\Gamma_1 \colon A_1=\bigoplus_{g \in G_1} A_1^{(g)},\qquad \Gamma_2
\colon A_2=\bigoplus_{g \in G_2} A_2^{(g)}\end{equation} be two gradings where $G_1$
and $G_2$ are groups and $A_1$ and $A_2$ are algebras.

The most restrictive case is when we require that both grading groups
coincide:

\begin{definition}[{e.g. \cite[Definition~1.15]{ElduqueKochetov}}]
	\label{DefGradedIsomorphism}
	Gradings~(\ref{EqTwoGroupGradings}) are \textit{isomorphic} if $G_1=G_2$ and there exists an isomorphism $\varphi \colon A_1 \mathrel{\widetilde\to} A_2$
	of algebras such that $\varphi\left(A_1^{(g)}\right)=A_2^{(g)}$
	for all $g\in G_1$.
\end{definition}

In this case we say that $A_1$ and $A_2$ are \textit{graded isomorphic}.

If one studies the graded structure of a graded algebra or its graded polynomial identities~\cite{AljaGia, AljaGiaLa,
	BahtZaiGradedExp, GiaLa, ASGordienko9}, then it is not really important by elements of which group the graded components are indexed.
A replacement of the grading group leaves both graded subspaces (and therefore graded ideals) graded. In the case of graded polynomial identities reindexing
the graded components leads only to renaming the variables.
Here we come naturally to the notion of equivalence of gradings.

\begin{definition}[{e.g. \cite[Definition~1.14]{ElduqueKochetov}}]\label{DefGradedEquivalence}
	Gradings~(\ref{EqTwoGroupGradings}) are \textit{equivalent}, if there
	exists an isomorphism $\varphi \colon A_1 \mathrel{\widetilde\to}A_2$
	of algebras such that for every $g_1\in G_1$ with $A_1^{(g_1)}\ne 0$ there
	exists $g_2\in G_2$ such that $\varphi\left(A_1^{(g_1)}\right)=A_2^{(g_2)}$.
\end{definition}
\begin{remark}
	In~\cite{ASGordienko18Schnabel} the term ``weak equivalence''  was used instead of simply ``equivalence''.
\end{remark}

Obviously, if gradings are isomorphic, then they are equivalent.
It is important to notice that the converse is not true.

However, if gradings~(\ref{EqTwoGroupGradings}) are equivalent
and $\varphi \colon A_1 \mathrel{\widetilde\to}A_2$ is the corresponding isomorphism of algebras,
then $\Gamma_3 \colon A_1=\bigoplus_{g \in G_2} \varphi^{-1}\left( A_2^{(g)}\right)$ is a $G_2$-grading on $A_1$ isomorphic
to $\Gamma_2$ and the grading $\Gamma_3$ is obtained from $\Gamma_1$ just by reindexing the homogeneous components.
Therefore, when gradings~(\ref{EqTwoGroupGradings}) are equivalent,
we can identify $\Gamma_2$ and $\Gamma_3$ via $\varphi$ and assume that we have gradings on the same algebra.
In this case we say that $\Gamma_1$ \textit{can be regraded} by~$G_2$.

If $A_1=A_2$ and $\varphi$ in Definition~\ref{DefGradedEquivalence} is the identity map, we say that $\Gamma_1$ and $\Gamma_2$ are \textit{realizations
	of the same grading} on $A$ as, respectively, a $G_1$- and an $G_2$-grading.

For a grading $\Gamma \colon A=\bigoplus_{g \in G} A^{(g)}$, we denote by $\supp \Gamma := \lbrace g\in G \mid A^{(g)}\ne 0\rbrace$ its support.

\begin{remark}Each equivalence between gradings $\Gamma_1$ and $\Gamma_2$
	induces a bijection $\supp \Gamma_1 \mathrel{\widetilde\to} \supp \Gamma_2$.
\end{remark}

Each group grading on an algebra can be realized as a $G$-grading for many different groups $G$, however it turns out that there is one distinguished group among them~\cite[Definition~1.17]{ElduqueKochetov}, \cite{PZ89}.

\begin{definition}\label{DefUniversalGrpGrd}
	Let $\Gamma$ be a group grading on an algebra $A$. Suppose that $\Gamma$ admits a realization
	as a $G_\Gamma$-grading for some group $G_\Gamma$. Denote by $\varkappa_\Gamma$ the corresponding embedding
	$\supp \Gamma \hookrightarrow G_\Gamma$. We say that $(G_\Gamma,\varkappa_\Gamma)$ is the \textit{universal group of the grading $\Gamma$} if for any realization of $\Gamma$ as a grading by a group $G$
	with $\psi \colon \supp \Gamma \hookrightarrow G$ there exists
	a unique homomorphism $\varphi \colon G_\Gamma \to G$ such that the following diagram is commutative:
	$$\xymatrix{ \supp \Gamma \ar[r]^(0.6){\varkappa_\Gamma} \ar[rd]_\psi & G_\Gamma \ar@{-->}[d]^\varphi \\
		& G
	}
	$$
\end{definition}

\begin{remark} 
	For each grading $\Gamma$ one can define a category $\mathcal C_\Gamma$
	where the objects are all pairs $(G,\psi)$ such that $G$ is a group and $\Gamma$ can be realized
	as a $G$-grading with $\psi \colon \supp \Gamma \hookrightarrow G$ being the embedding of the support.
	In this category the set of morphisms between $(G_1,\psi_1)$ and $(G_2,\psi_2)$ consists
	of all group homomorphisms $f \colon G_1 \to G_2$ such that the diagram below is commutative:
	$$\xymatrix{ \supp \Gamma \ar[r]^(0.6){\psi_1} \ar[rd]_{\psi_2} & G_1 \ar[d]^f \\
		& G_2
	}
	$$
	Then $(G_\Gamma,\varkappa_\Gamma)$ is the initial object of $\mathcal C_\Gamma$.
\end{remark}
\begin{remark} 
	It is easy to see that if $\Gamma \colon A = \bigoplus_{g\in\supp \Gamma} A^{(g)}$ is a group grading,
	then the universal group $G_\Gamma$ of the grading $\Gamma$ is isomorphic to $\mathcal F_{[\supp \Gamma]}/N$ where $\mathcal F_{[\supp \Gamma]}$ is the free group on the set $[\supp \Gamma]:=\lbrace [g] \mid g \in \supp \Gamma \rbrace$ and $N$ is the normal closure of the words $[g][h][t]^{-1}$ for pairs $g,h \in \supp \Gamma$
	such that $A^{(g)}A^{(h)}\ne 0$ where $t\in \supp\Gamma$ is defined by $A^{(g)}A^{(h)}\subseteq A^{(t)}$.
\end{remark}

The classification of gradings on a given algebra $A$ up to equivalence may seem coarser than the classification
up to isomorphism. However, if we know the classification up to equivalence
and calculate the universal groups of the corresponding gradings, we can recover the classification
up to isomorphism by restricting the grading groups to the subgroups generated by the supports.

\section{Quantum symmetries}

\subsection{(Co)module algebras}\label{Subsection(Co)moduleAlgebras}

We refer the reader to~\cite{Abe,DNR, KasselQuantumGroups, Montgomery}
for an account on Hopf algebras and their (co)actions on algebras.

Let $A$ be a (not necessarily associative) algebra and let $H$ be a Hopf algebra over a field~$\mathbbm{k}$.
A $\mathbbm{k}$-linear map $\psi \colon H \otimes A \to A$ is called an \textit{$H$-action} on $A$
if $\psi$ defines on $A$ a structure of a left $H$-module and
\begin{equation}\label{EqModCompat} h(ab)=(h_{(1)}a)(h_{(2)}b) \text{
	for all }a,b\in A,\ h\in H\end{equation}
where we use Sweedler's notation $\Delta h = h_{(1)}\otimes h_{(2)}$ (the summation sign is omitted)
for the comultiplication $\Delta \colon H \to H \otimes H$. The algebra $A$ endowed with $\psi$
is called an \textit{$H$-module algebra}. If there exists
an identity element $1_A \in A$ such that $h 1_A =\varepsilon(h)1_A$
for all $h\in H$, then $A$ and $\psi$ will be called a \textit{unital} $H$-module algebra
and a \textit{unital} $H$-action, respectively.

\begin{example} Let $G$ be a group and let $\mathbbm k$ be a field. Denote by $\mathbbm k G$ the group algebra of~$G$
	over $\mathbbm k$. Recall that $\mathbbm k G$ is a Hopf algebra where the comultiplication $\Delta \colon \mathbbm k G \to \mathbbm k G \otimes \mathbbm k G$, the counit $\varepsilon \colon \mathbbm k G \to \mathbbm k G$
	and the antipode $S \colon \mathbbm k G \to \mathbbm k G$ are the linear maps defined on the elements of the standard basis of  $\mathbbm k G$ as follows: $\Delta g := g\otimes g$, $\varepsilon(g) := 1$, $Sg := g^{-1}$ for all $g\in G$.
	Then a $\mathbbm k G$-module algebra is precisely an algebra on which the group $G$ is acting by automorphisms.
\end{example}

\begin{example} Let $\mathfrak g$ be a Lie algebra over a field $\mathbbm k$. Recall that the universal enveloping algebra $U(\mathfrak g)$ of $\mathfrak g$ is a Hopf algebra where the comultiplication $\Delta \colon U(\mathfrak g) \to U(\mathfrak g) \otimes U(\mathfrak g)$ and the counit $\varepsilon \colon U(\mathfrak g) \to U(\mathfrak g)$
	are the unital algebra homomorphisms
	and the antipode $S \colon U(\mathfrak g) \to U(\mathfrak g)$
	is the unital algebra anti-homomorphism defined on the standard generators of $U(\mathfrak g)$ as follows: $\Delta v := v\otimes 1 + 1 \otimes v$, $\varepsilon(v) := 0$, $Sv := -v$ for all $v\in \mathfrak g$.
	Then a $U(\mathfrak g)$-module algebra is precisely an algebra on which the Lie algebra $\mathfrak g$ is acting by derivations.
\end{example}

Again, let $A$ be a (not necessarily associative) algebra and let $H$ be a Hopf algebra over a field~$\mathbbm{k}$.
A $\mathbbm{k}$-linear map $\rho \colon A \to A \otimes H$ is called an \textit{$H$-coaction} on $A$
if $\rho$ defines on $A$ a structure of a right $H$-comodule
and $$\rho(ab)=  a_{(0)} b_{(0)} \otimes a_{(1)} b_{(1)}$$
for all $a,b \in A$ where we use Sweedler's notation $\rho(a)=a_{(0)}\otimes a_{(1)}$
and the summation sign is again omitted. In other words,
$\rho \colon A \to A \otimes H$ is an algebra homomorphism.
 The algebra $A$ endowed with $\rho$
is called an \textit{$H$-comodule algebra}. If there exists
an identity element $1_A \in A$ such that $\rho(1_A)=1_A\otimes 1_{H}$, then $A$ and $\rho$
 will be called a \textit{unital} $H$-comodule algebra  and a \textit{unital} $H$-coaction, respectively.

\begin{example} Let  $\Gamma \colon A=\bigoplus_{g \in G} A^{(g)}$ be a grading on an algebra $A$ over a field $\mathbbm k$ by a group $G$.
	Define the linear map $\rho \colon A \to A \otimes \mathbbm k G$ by
	$\rho(a):= a \otimes g$ for all $a\in A^{(g)}$ and $g\in G$. Then $\rho$ is a $\mathbbm k G$-coaction on $A$.
	Conversely, if $\rho \colon A \to A \otimes \mathbbm k G$ is a $\mathbbm k G$-coaction on some algebra $A$, then $\Gamma \colon A=\bigoplus_{g \in G} A^{(g)}$ is a $G$-grading where $A^{(g)} := \lbrace a\in A
	\mid \rho(a)= a \otimes g  \rbrace$ for all $g\in G$.
\end{example}
\begin{example}  Recall that the algebra $\mathcal O (G)$ of regular functions
	on an affine algebraic group $G$  over a field $\mathbbm k$ is a Hopf algebra where the unital algebra homomorphisms $\Delta \colon \mathcal O (G) \to \mathcal O (G) \otimes \mathcal O (G)$ and $S \colon \mathcal O (G) \to \mathcal O (G)$ correspond to the multiplication $G \times G \to G$
	and taking the inverse $(-)^{-1} \colon G \to G$, respectively. The counit $\varepsilon \colon \mathcal O (G) \to \mathbbm k$
	is defined by $\varepsilon(f):=f(1_G)$ for all $f\in \mathcal O (G)$.
	Suppose  $G \times X \to X$ is a $G$-action on an affine algebraic variety $X$ by morphisms.
	Then the corresponding unital algebra homomorphism $\rho \colon \mathcal O (X) \to \mathcal O (X) \otimes \mathcal O (G)$
	defines on the algebra $\mathcal O (X)$ of regular functions on $X$ a structure of an $\mathcal  O (G)$-comodule algebra.
	Conversely, every coaction $\rho \colon \mathcal O (X) \to \mathcal O (X) \otimes \mathcal O (G)$
	corresponds to a $G$-action on $X$ by morphisms, see the details, e.g., in~\cite[Chapter~4]{Abe}.
\end{example}

Hopf algebra (co)actions on an algebra $A$ are called \textit{quantum symmetries} of $A$.
If $A$ is a unital associative algebra, then we assume (co)actions to be unital too.
 If $X$ is an affine algebraic variety, then quantum symmetries of $\mathcal O(X)$ are called 
\textit{quantum symmetries} of $X$.

\subsection{Equivalence of (co)actions}

Let $\rho \colon A \to B \otimes Q$ be a linear map for some vector spaces $A,B,Q$ over a field $\mathbbm k$.
Define the map $\rho^\vee \colon Q^* \otimes A \to B $
by $\rho^\vee(\lambda \otimes a):= \sum\limits_{i=1}^k\lambda(q_i)b_i$
for all $a\in A$ and $\lambda \in Q^*$
where $\rho(a)=  \sum\limits_{i=1}^k b_i \otimes q_i$, $k\in\mathbb Z_+$, $b_i \in B$, $q_i \in Q$.

Let $\psi \colon H \otimes A \to A$ be an $H$-action for some Hopf algebra $H$ and an algebra $A$ over a field~$\mathbbm k$.
Then the unital subalgebra $\cosupp \psi := \lbrace \psi(h \otimes (-)) \mid h \in H \rbrace \subseteq \End_{\mathbbm k}(A)$ is called  the \textit{cosupport} of $\psi$.

Let $\rho \colon A \to A \otimes H$ be an $H$-coaction for some Hopf algebra $H$ and an algebra $A$ over a field~$\mathbbm k$.
Then the unital subalgebra $\cosupp \rho := \lbrace \rho^\vee(\lambda \otimes (-)) \mid \lambda \in H^* \rbrace \subseteq \End_{\mathbbm k}(A)$ is called  the \textit{cosupport} of $\rho$.

\begin{example}
	If $\Gamma \colon A=\bigoplus_{g \in G} A^{(g)}$ is a grading on an algebra $A$ by a group $G$
	and $\rho \colon A \to A \otimes \mathbbm k G$ is the corresponding $\mathbbm k G$-coaction,
	then $\cosupp \rho$ consists of all linear operators $A\to A$ that are scalar on each graded component $A^{(g)}$.
	Hence $\cosupp \rho$ is isomorphic to the algebra of all functions $\supp G \to \mathbbm k$ with the pointwise operations.
\end{example}

\begin{definition} Let $\psi_i \colon H_i \otimes A_i \to A_i$ be $H_i$-actions for some Hopf algebras $H_i$ and algebras $A_i$ over a field~$\mathbbm k$. We say that an algebra isomorphism
	$\varphi \colon A_1 \mathrel{\widetilde\to} A_2$ is an \textit{equivalence} of $\psi_1$ and $\psi_2$  if
	$\tilde\varphi \left( \cosupp \psi_1 \right) =  \cosupp \psi_2$
	where the isomorphism $\tilde \varphi \colon \End_{\mathbbm k}(A_1) \mathrel{\widetilde\to} \End_{\mathbbm k}(A_2)$
	is defined by $\tilde \varphi(f) := 
	\varphi f \varphi^{-1}$ for all $f\in \End_{\mathbbm k}(A_1)$.
\end{definition}
 
\begin{definition} Let $\rho_i \colon A_i \to A_i \otimes H_i$ be $H_i$-coactions for some Hopf algebras $H_i$ and algebras $A_i$ over a field~$\mathbbm k$. We say that an algebra isomorphism
	$\varphi \colon A_1 \mathrel{\widetilde\to} A_2$ is an \textit{equivalence} of $\rho_1$ and $\rho_2$  if
	$\tilde\varphi \left( \cosupp \rho_1 \right) =  \cosupp \rho_2$.
\end{definition}

If $\varphi \colon A_1 \mathrel{\widetilde\to} A_2$ is an equivalence of some (co)actions, then we will call the (co)actions \textit{equivalent} via $\varphi$. Again, we can always identify $A_1$ and $A_2$ and assume that $\varphi$ is the identity map.

It turns out that the notion of equivalence introduced above is indeed a generalization of the corresponding notion for gradings:
\begin{theorem}[{\cite[Theorem~3.7]{ASGordienko20ALAgoreJVercruysse}}]\label{TheoremGradEquivHopfEquiv}
	Let $\Gamma_1 \colon A_1=\bigoplus\limits_{g \in G_1} A_1^{(g)}$ and $\Gamma_2
	\colon A_2=\bigoplus\limits_{g \in G_2} A_2^{(g)}$ be two gradings where $G_1$
	and $G_2$ are groups and $A_1$ and $A_2$ are algebras over a field $\mathbbm{k}$.
	Denote by $\rho_i \colon A_i \to A_i \otimes \mathbbm{k} G_i$ the corresponding $\mathbbm{k} G_i$-coactions.
	Then an algebra isomorphism $\varphi \colon A_1 \mathrel{\widetilde\to} A_2$ is an equivalence of $\rho_1$ and $\rho_2$  if
	and only if $\varphi$ is an equivalence of $\Gamma_1$ and $\Gamma_2$.
\end{theorem}	

\subsection{$V$-universal (co)acting Hopf algebras}

Let $A$ be an algebra over a field $\mathbbm k$ and let $V \subseteq \End_{\mathbbm k}(A)$ be some unital
subalgebra.

\begin{definition}
A pair $\left({}_\square \mathcal H(A,V), \psi^{\mathbf{Hopf}}_{A,V}\right)$
is called a \textit{$V$-universal acting Hopf algebra} if ${}_\square \mathcal H(A,V)$ is a Hopf algebra, $\psi^{\mathbf{Hopf}}_{A,V} \colon {}_\square \mathcal H(A,V) \otimes A \to A$ is an ${}_\square \mathcal H(A,V)$-action,
$\cosupp \psi^{\mathbf{Hopf}}_{A,V} \subseteq V$ and for every $H$-action $\psi \colon H \otimes A \to A$,
where $H$ is a Hopf algebra and $\cosupp \psi \subseteq V$, there exists a unique Hopf algebra homomorphism $\varphi$
making the diagram below commutative:
$$\xymatrix{ H \otimes A \ar[r]^(0.6){\psi} \ar@{-->}[d]_{\varphi \otimes \id_A} & A \\
	{}_\square \mathcal{H}(A,V) \otimes A  \ar[ru]_(0.6){\psi^{\mathbf{Hopf}}_{A,V}}  }.$$ 
\end{definition}

The $V$-universal acting Hopf algebra exists for every $V \subseteq \End_{\mathbbm k}(A)$, see~\cite[Section~4.1]{ASGordienko21ALAgoreJVercruysse}.

Let $\psi \colon H \otimes A \to A$ be a Hopf algebra action. Then ${}_\square \mathcal{H}(A,\cosupp \psi)$ is called the \textit{universal Hopf algebra} of the action $\psi$.

\begin{definition}
	A pair $\left(\mathcal H^\square(A,V), \rho^{\mathbf{Hopf}}_{A,V}\right)$
	is called a \textit{$V$-universal coacting Hopf algebra} if $\mathcal H^\square(A,V)$ is a Hopf algebra, $\rho^{\mathbf{Hopf}}_{A,V} \colon A \to A \otimes \mathcal{H}^\square(A,V)$ is an $\mathcal H^\square(A,V)$-coaction,
	$\cosupp \rho^{\mathbf{Hopf}}_{A,V} \subseteq V$ and for every $H$-coaction $\rho \colon A \to A \otimes H$,
	where $H$ is a Hopf algebra and $\cosupp \rho \subseteq V$, there exists a unique Hopf algebra homomorphism $\varphi$
	making the diagram below commutative:
	$$\xymatrix{ A \ar[rd]_\rho \ar[r]^(0.3){\rho^{\mathbf{Hopf}}_{A,V}} & A \otimes \mathcal{H}^\square(A,V)
		\ar@{-->}[d]^{\id_A \otimes \varphi} \\
		& A \otimes H }$$ 
\end{definition}

Let $U$ and $W$ be $\mathbbm k$-vector spaces. Recall that the \textit{finite topology} on $\Hom_{\mathbbm k}(U,W)$
is the topology defined using the base $$\lbrace f \in \Hom_{\mathbbm k}(U,W) \mid f(u_j)=w_j
\text{ for all } 1\leqslant j\leqslant m  \rbrace$$ where $u_j \in U$, $w_j \in W$, $1\leqslant j \leqslant m$,
$m\in\mathbb Z_+$.

The $V$-universal coacting Hopf algebra $\mathcal H^\square(A,V)$ exists for every $V \subseteq \End_{\mathbbm k}(A)$
\textit{pointwise finite dimensional}, i.e. such that $\dim\lbrace f(a) \mid f\in  V \rbrace <+\infty$
for every $a\in A$, and closed in the finite topology~\cite[Theorem~4.10]{ASGordienko21ALAgoreJVercruysse}.
 If $A$ is finite dimensional, $\mathcal H^\square(A,V)$ exists for every $V \subseteq \End_{\mathbbm k}(A)$ since in this case all subspaces of $\End_{\mathbbm k}(A)$ are pointwise finite dimensional and closed in the finite topology.

\begin{remark}\label{RemarkUniversalPropertyOfACoaction}
Let $\rho \colon A \to A \otimes H$ be a Hopf algebra coaction. Then $\mathcal H^\square(A,\cosupp \rho)$ is called the \textit{universal Hopf algebra} of the coaction $\rho$.
By~\cite[Theorem~2.11]{ASGordienko21ALAgoreJVercruysse}, $\cosupp \rho$ is pointwise finite dimensional and closed in the finite topology, i.e. $\mathcal H^\square(A,\cosupp \rho)$ always exists.
Moreover, by the universal property of its coaction $\rho^{\mathbf{Hopf}}_{A,\cosupp \rho} \colon A \to A \otimes \mathcal{H}^\square(A,\cosupp \rho)$, we have $\cosupp \rho^{\mathbf{Hopf}}_{A,\cosupp \rho} = \cosupp \rho$.
 Therefore, in order to prove that a Hopf algebra $H_0$ together with its coaction  $\rho_0 \colon A \to A \otimes H_0$,
 where $\cosupp \rho_0 = \cosupp \rho$, is a universal Hopf algebra of $\rho$, it is sufficient to check that there exists a unique Hopf algebra homomorphism
 $\varphi \colon H_0 \to H_1$ such that $\rho_1 = (\id_A \otimes \varphi)\rho_0$
 for all Hopf algebra coactions $\rho_1 \colon A \to A \otimes H_1$ where $\cosupp \rho_1=\cosupp \rho$.
 The universal properties of both $H_0$ and $\mathcal H^\square(A,\cosupp \rho)$ then would imply that there exists a Hopf algebra isomorphism of $H_0$ and $\mathcal H^\square(A,\cosupp \rho)$ compatible with their coactions,
 and $\rho_0$ is universal not only among $\rho_1$ with $\cosupp \rho_1=\cosupp \rho$, but among
 $\rho_1$ with $\cosupp \rho_1 \subseteq \cosupp \rho$ too.
\end{remark}

Let $A$ be a unital algebra over a field $\mathbbm k$. Denote by $V_1(A)$ the unital subalgebra of 
$ \End_{\mathbbm k}(A)$ that consists of all $\mathbbm k$-linear operators $A \to A$ for which $1_A$ is an eigenvector.

Proposition~\ref{PropositionUnitalCoaction} below is dual to~\cite[Proposition~5.3]{ASGordienko20ALAgoreJVercruysse}:
\begin{proposition}\label{PropositionUnitalCoaction}
	Let $\rho \colon A \to A \otimes H$ be a coaction for a unital algebra $A$ and a Hopf algebra $H$
	over a field $\mathbbm k$
	such that $\cosupp \rho \subseteq V_1(A)$. Then $\rho$ is a unital coaction.
\end{proposition}
\begin{proof} If $A=0$, then there is nothing to prove, so below we assume that $A\ne 0$ and therefore $1_A\ne 0$.
	
	Consider the decomposition of $\rho(1_A)$ with respect to bases of $A$ and $H$ where the basis of $A$ includes $1_A$. If any basis elements of $A$ different from $1_A$ were to appear in this decomposition, then, considering the linear functions $\lambda \in H^*$  dual to the basis elements of $H$, it would be possible to find some $\lambda \in H^*$
	such that $1_A$ would not be an eigenvector for $(\id_A \otimes \lambda)\rho$.
	Hence the condition~$\cosupp \rho \subseteq V_1(A)$ implies that $\rho(1_A)=1_A \otimes h$ for some $h\in H$. 
	Then $1_A \otimes \Delta h = (\rho \otimes \id_H)\rho(1_A)= 1_A \otimes h \otimes h$.
	By the counit axiom, $1_A = \varepsilon(h)1_A$, i.e. $\varepsilon(h)=1$ and $h\ne 0$.
	Thus $h$ is group-like and, in particular, invertible. At the same time 
	$$1_A \otimes h = \rho(1_A)=\rho(1_A^2)=1_A^2 \otimes h^2= 1_A \otimes h^2.$$
	Hence $h^2=h$ and $h=1_H$. Therefore, $\rho$ is a unital coaction.	
\end{proof}	

The proposition above implies that the $ V_1(A)$-universal coacting Hopf algebra $\mathcal H^\square(A):=\mathcal H^\square(A,V_1(A))$ (if it exists) is
universal among all unital coactions on $A$ and therefore coincides with
   \textit{Manin (ungraded) universal coacting Hopf algebra} $\underline{\mathrm{aut}}(A)$~\cite{Manin} (see also
   \cite[Corollary~2.8.5]{TheoRadMichVdB}).
In~\cite[Example~4.20]{ASGordienko21ALAgoreJVercruysse} the first example of an algebra $A$ was constructed such that $\mathcal H^\square(A)$ did not exist. In Theorem~\ref{TheoremYuIManinHopfAlgebraStraightLineDoesNotExist}
below we show that even $\mathcal H^\square(\mathbbm k[x])$ does not exist
where $\mathbbm k[x]$ is the algebra of polynomials in $x$, i.e. the algebra of regular functions on the straight line.

Theorem~\ref{TheoremUniversalHopfOfAGradingIsJustUniversalGroup} below relates $V$-universal coacting Hopf algebras
and universal groups of gradings:
\begin{theorem}[{\cite[Theorem~4.11]{ASGordienko20ALAgoreJVercruysse}}]\label{TheoremUniversalHopfOfAGradingIsJustUniversalGroup}
	Let $\Gamma \colon A = \bigoplus_{g\in G} A^{(g)}$ be a grading on an algebra $A$ by a group $G$. Denote by $\rho \colon A \to A \otimes \mathbbm kG$ the corresponding $\mathbbm kG$-coaction.
	Let $G_\Gamma$ be the universal group of $\Gamma$ and let $\rho_\Gamma \colon A \to A \otimes \mathbbm kG_\Gamma$
	be the corresponding $\mathbbm kG_\Gamma$-coaction.
	Then $\cosupp \rho_\Gamma = \cosupp \rho$ and
	 $(\mathbbm kG_\Gamma, \rho_\Gamma)$ is the $(\cosupp \rho)$-universal coacting Hopf algebra.
\end{theorem}

Given a unital associative algebra $A$ denote by $A^\circ$ its finite (or Sweedler) dual.
Recall that $A^\circ$ is the subspace of $A^*$ which consists of all linear functions
$\lambda \in A^*$ such that $\Ker \lambda \supseteq I$ for some two-sided ideal $I \subseteq A$, $\dim(A/I) < +\infty$.
($I$ depends on $\lambda$.)
Denote by $\varkappa_A \colon A^\circ \hookrightarrow A^*$ the corresponding embedding.
If $H$ is a Hopf algebra, then $H^\circ$ is a Hopf algebra too. If $H$ is finite dimensional,
then $H^\circ=H^*$.

An important role in the classification of quantum symmetries is played by the Duality Theorem below:
\begin{theorem}[{\cite[Theorem~4.15]{ASGordienko21ALAgoreJVercruysse}}]\label{TheoremUnivHopf(Co)actDuality}
	Let $A$ be an algebra over a field $\mathbbm k$ and let $V\subseteq \End_{\mathbbm k}(A)$ be a unital pointwise finite dimensional subalgebra  closed in the finite topology.
	Then $\bigl(\rho^\mathbf{Hopf}_{A,V}\bigr)^\vee(\varkappa_{\mathcal{H}^\square(A,V)} \otimes \id_A)
	\colon \mathcal{H}^\square(A,V)^\circ \otimes A \to A$
	is an $\mathcal{H}^\square(A,V)^\circ$-action and	
	the unique homomorphism $\theta^\mathbf{Hopf} \colon \mathcal{H}^\square(A,V)^\circ \to {}_\square \mathcal{H}(A,V) $ of Hopf algebras  making
	the diagram $$\xymatrix{ {}_\square \mathcal{H}(A,V) \otimes A \ar[rr]^(0.6){\psi^\mathbf{Hopf}_{A,V}} & 
		& A	\\
		\mathcal{H}^\square(A,V)^\circ \otimes A \ar[rr]^{\varkappa_{\mathcal{H}^\square(A,V)} \otimes \id_A}
		\ar@{-->}[u]^{\theta^\mathbf{Hopf} \otimes \id_A} & \qquad &  \mathcal{H}^\square(A,V)^* \otimes A \ar[u]_{\bigl(\rho^\mathbf{Hopf}_{A,V}\bigr)^\vee} } $$  commutative is an isomorphism.
\end{theorem}

Categorical versions of the definitions, existence and duality theorems mentioned in this section can be found in~\cite{ASGordienko24ALAgoreJVercruysse,ASGordienko25ALAgoreJVercruysse}.

\subsection{Outline of the classification strategy}\label{SectionClassificationOutline}
The Duality Theorem~\ref{TheoremUnivHopf(Co)actDuality} suggests using the following strategy in the classification of quantum symmetries.

If $\psi \colon H \otimes A \to A$ is a Hopf algebra action on a finite dimensional algebra $A$, then
it factorizes through the action $\psi^{\textbf{Hopf}}_{A,\cosupp \psi} $ of the universal Hopf algebra 
 ${}_\square \mathcal{H}(A,\cosupp \psi)$ of $\psi$. Hence $\cosupp \psi \subseteq  \cosupp \psi^{\textbf{Hopf}}_{A,\cosupp \psi}$. At the same time, the definition of  $\psi^{\textbf{Hopf}}_{A,V} $
 implies that $\cosupp \psi^{\textbf{Hopf}}_{A,V} \subseteq V$ for every unital subalgebra $V\subseteq \End_{\mathbbm k}(A)$. Therefore,
\begin{equation*}\begin{split}\cosupp \psi =  \cosupp \psi^{\textbf{Hopf}}_{A,\cosupp \psi} 
= \cosupp \bigl(\rho^\mathbf{Hopf}_{A,\cosupp \psi}\bigr)^\vee(\varkappa_{\mathcal{H}^\square(A,\cosupp \psi)} \otimes \id_A) \subseteq \\ \cosupp \rho^{\textbf{Hopf}}_{A,\cosupp \psi} \subseteq \cosupp \psi,\end{split}\end{equation*}
i.e. $\cosupp \rho^{\textbf{Hopf}}_{A,\cosupp \psi}= \cosupp \psi$
and there always exists a coaction with the same support as an action. Therefore,
when studying quantum symmetries on finite dimensional algebras, one should start with coactions. Another reason for this is that universal coacting Hopf algebras
are usually easier to describe (see~\cite[Theorem~4.8]{ASGordienko20ALAgoreJVercruysse} and~\cite[Theorem~3.16 and Section~4.2]{ASGordienko21ALAgoreJVercruysse}) than the universal acting Hopf algebras.

Hence, given a finite dimensional algebra $A$, we first determine all unital subalgebras
$V \subseteq \End_{\mathbbm k}(A)$ that are cosupports of Hopf algebra coactions.
For such $V$ the universal property of $\mathcal{H}^\square(A,V)$
implies that $\cosupp \rho^{\textbf{Hopf}}_{A,V} = V$.
 Then we try to find a Hopf algebra action $\psi \colon H \otimes A \to A$ such that
 $\cosupp \psi = V$. If such $\psi$ exists, then $\cosupp \psi^{\textbf{Hopf}}_{A,V} = V$
 and ${}_\square \mathcal{H}(A,V) \cong \mathcal{H}^\square(A,V)^\circ$.
 If such $\psi$ does not exist, then $\cosupp \psi^{\textbf{Hopf}}_{A,V} =: V_0 \subsetneqq V$
 and $$\psi^{\textbf{Hopf}}_{A,V} = \psi^{\textbf{Hopf}}_{A,V_0} = \bigl(\rho^\mathbf{Hopf}_{A,V_0}\bigr)^\vee(\varkappa_{\mathcal{H}^\square(A,V_0)} \otimes \id_A),$$
 i.e. the corresponding universal action appears for another choice of $V$.
 
 In Section~\ref{SectionTwoPoints} and~\ref{SectionDualNumbers} we give two examples where
 this strategy makes it possible to completely classify the quantum symmetries.

\section{Set of two points}\label{SectionTwoPoints}

The algebra of regular functions on the set of two points is isomorphic to the direct product $\mathbbm k \times \mathbbm k$
of two copies of the base field $\mathbbm k$.
 Note that $1_{\mathbbm k \times \mathbbm k} = (1,1)$. Let $e:=(1,0)$.

\begin{lemma}\label{LemmaTwoPointsCoactRelations} Let  $\rho \colon {\mathbbm k \times \mathbbm k} \to (\mathbbm k \times \mathbbm k) \otimes H$ be an arbitrary unital $H$-coaction where $H$ is a Hopf algebra. Define $c,h\in H$ by $\rho(e) =  1_{\mathbbm k \times \mathbbm k} \otimes h + e \otimes c$.
	Then $c$ is invertible, $$h^2=h,\quad hc+ch+c^2 = c,$$   $$\Delta c= c\otimes c,\quad \Delta h = h \otimes c + 1_H \otimes h,$$
	$$\varepsilon(c)=1,\quad \varepsilon(h) = 0,$$ $$Sc = c^{-1},\quad Sh = -hc^{-1}.$$
\end{lemma}
\begin{proof} By the unitality of $\rho$, we have $\rho(1_{\mathbbm k \times \mathbbm k})=1_{\mathbbm k \times \mathbbm k}\otimes 1_H$. 
	
	The equality $e^2=e$ implies that
	$$1_{\mathbbm k \times \mathbbm k} \otimes h + e \otimes c = \rho(e) = \rho\left(e^2 \right) = 
	\rho(e)^2=
	(1_{\mathbbm k \times \mathbbm k} \otimes h + e \otimes c)^2 = 1_{\mathbbm k \times \mathbbm k} \otimes h^2 + e \otimes (hc+ch+c^2),$$
	i.e.
	$h^2=h$ and $hc+ch+c^2 = c$.
	
	At the same time, from $(\id_{\mathbbm k \times \mathbbm k} \otimes \Delta)\rho = (\rho \otimes \id_H)\rho$
	it follows that
	$$1_{\mathbbm k \times \mathbbm k} \otimes \Delta h + e \otimes \Delta c = 1_{\mathbbm k \times \mathbbm k} \otimes 1_H \otimes h + 1_{\mathbbm k \times \mathbbm k} \otimes h \otimes c + e \otimes c \otimes c,$$
	i.e. $\Delta h = h \otimes c + 1_H \otimes h$, $\Delta c= c\otimes c$.
	
	Finally, by  the counitality of a comodule,
	$e = \varepsilon(h)1_{\mathbbm k \times \mathbbm k} + \varepsilon(c) e$,
	i.e. $\varepsilon(h) = 0$ and $\varepsilon(c)=1$.

Note that $c$ is a group-like element. Hence $c$ is invertible and $Sc = c^{-1}$. Moreover,
$$0=\varepsilon(h)1_H = (Sh_{(1)}) h_{(2)} = (Sh)c+h$$
implies $Sh = -hc^{-1}$.
\end{proof}

\begin{theorem}\label{TheoremTwoPointsCoactUniversal} Consider the free unital associative algebra $\mathbbm k \langle c,d, h \rangle$ over a field $\mathbbm k$ with
	free generators $c,d, h$, i.e. the algebra of polynomials in the non-commuting variables $c,d, h$.
	Denote by $I$ the ideal of $\mathbbm k \langle c,d, h \rangle$ generated by the elements
	$$cd-1,\quad dc-1,\quad h^2-h,\quad hc+ch+c^2 - c.$$
	Define $\Delta \colon \mathbbm k \langle c,d, h \rangle \to  \mathbbm k \langle c,d, h \rangle \otimes  \mathbbm k \langle c,d, h \rangle$ and $\varepsilon \colon \mathbbm k \langle c,d, h \rangle \to  \mathbbm k$
	as the unital algebra homomorphisms such that $$\Delta c= c\otimes c,\quad \Delta d= d\otimes d,\quad \Delta h = h \otimes c + 1 \otimes h,$$
	$$\varepsilon(c)=\varepsilon(d)=1,\quad \varepsilon(h) = 0.$$
	Define $S \colon \mathbbm k \langle c,d, h \rangle \to \mathbbm k \langle c,d, h \rangle$ 
	as the unital algebra anti-homomorphism such that $$Sc = d,\ Sd = c,\ Sh = -hd.$$
	Then $$S I \subseteq I,\quad \varepsilon(I)=0,\quad \Delta I\subseteq
	 I \otimes \mathbbm k \langle c,d, h \rangle +
	 \mathbbm k \langle c,d, h \rangle \otimes I$$ and the algebra $\mathbbm k \langle c,d, h \rangle/I$ with the induced maps $\Delta, \varepsilon, S$ is a Hopf algebra. Finally, $\mathbbm k \langle c,d, h \rangle/I=\mathcal H^\square({\mathbbm k \times \mathbbm k})$, the Manin universal coacting Hopf algebra  on ${\mathbbm k \times \mathbbm k}$.
\end{theorem}
\begin{proof} Since $\Delta$ and $\varepsilon$ are defined as unital algebra homomorphisms,
	it is sufficient to check the coassociativity and the counit axioms on the generators $c,d,h$.
	An easy verification shows that $\mathbbm k \langle c,d, h \rangle$ is a coalgebra and therefore a bialgebra.
	
	Note that $$\Delta(cd-1)=cd \otimes cd - 1 \otimes 1 = (cd -1)\otimes cd + 1 \otimes (cd-1) \in I \otimes \mathbbm k \langle c,d, h \rangle +
	\mathbbm k \langle c,d, h \rangle \otimes I.$$
	The same is true for $(dc-1)$ too.
	Moreover,
	\begin{equation*}\begin{split}\Delta(h^2-h)=h^2 \otimes c^2 + h\otimes ch + h\otimes hc + 1\otimes h^2 - h \otimes c - 1 \otimes h
	\\=(h^2-h)\otimes c^2 + h\otimes(hc+ch+c^2-c) + 1\otimes(h^2-h) \in I \otimes \mathbbm k \langle c,d, h \rangle +
	\mathbbm k \langle c,d, h \rangle \otimes I,\end{split}\end{equation*}
	\begin{equation*}\begin{split}\Delta(hc+ch+c^2-c)=hc \otimes c^2 + c\otimes hc  + ch \otimes c^2 + c\otimes ch  + c^2 \otimes c^2 - c \otimes c
		\\= (hc+ch+c^2-c)\otimes c^2 + c\otimes (hc+ch+c^2-c) \in I \otimes \mathbbm k \langle c,d, h \rangle +
		\mathbbm k \langle c,d, h \rangle \otimes I.\end{split}\end{equation*}
Hence $I$ is a biideal and  $\mathbbm k \langle c,d, h \rangle/I$ is a bialgebra.

Note that since $S$ is an anti-homomorphism, in order to prove that $SI\subseteq I$,
it is sufficient to verify that the images under $S$ of generators of the ideal $I$  again belong to $I$.

The latter is indeed true. First, $S(cd-1)=dc-1 \in I$, $S(dc-1)=cd-1 \in I$.
Moreover, $$dh+hd-d+1=d(hc+ch+c^2-c)d+dh(1-cd)+(1-dc)hd+dc(1-cd)+(1-dc)-d(1-cd) \in I.$$
Hence
 $$S(h^2-h)=hdhd+hd = h(dh+hd-d+1)d-(h^2-h)d^2 \in I.$$
 Finally,
 $$S(hc+ch+c^2-c)=-dhd-hd^2 + d^2 - d = -(dh+hd-d+1)d \in I.$$
 Therefore, $SI \subseteq I$ and $S$ is well defined on $\mathbbm k \langle c,d, h \rangle/I$.
 An explicit verification shows that $S$ is indeed the antipode and $\mathbbm k \langle c,d, h \rangle/I$ is a Hopf algebra.
 
 Define the unital homomorphism $$\rho \colon {\mathbbm k \times \mathbbm k} \to (\mathbbm k \times \mathbbm k) \otimes \mathbbm k \langle c,d, h \rangle/I$$ such that $\rho(e):= 1_{\mathbbm k \times \mathbbm k} \otimes \bar h + e\otimes \bar c$
 where $\bar a := a + I$ for all $a\in \mathbbm k \langle c,d, h \rangle$.
 Then $\rho$ is a unital coaction. Lemma~\ref{LemmaTwoPointsCoactRelations} implies that $\rho$ is universal among all unital coactions on ${\mathbbm k \times \mathbbm k}$.\end{proof}	

Recall that the \textit{Sweedler algebra} $H_4$ over a field $\mathbbm k$ is the Hopf algebra that is generated as a unital algebra by elements $c$ and $v$ where $c^2=1$, $v^2=0$, $vc=-cv$, $H_4 = \langle 1, c, v, cv \rangle_{\mathbbm k}$, $Sc=c$, $Sv=-cv$, $\Delta c = c\otimes c$, $\Delta v = c\otimes v + v\otimes 1$,
$\varepsilon(c)=1$, $\varepsilon(v)=0$. Note that if $\ch \mathbbm k = 2$, then $-1=1$ is no longer a primitive root of unity, but still $(c\otimes v + v\otimes 1)^2=0$, i.e. the Sweedler algebra is defined for any $\ch \mathbbm k$.

\begin{theorem}\label{TheoremTwoPointsCoactUpperTriangularFinDim} Let $\mathbbm k \langle c,d, h \rangle$, $I$, $\Delta$, $\varepsilon$, $S$ be the same
	as in Theorem~\ref{TheoremTwoPointsCoactUniversal}.
	Denote by $I_1$ the ideal of $\mathbbm k \langle c,d, h \rangle$ generated by the element
	$c^2-1$. Then $\mathbbm k \langle c,d, h \rangle/(I+I_1)$ is a four-dimensional Hopf algebra
	with the basis $\bar 1, \bar c, \bar h, \bar c \bar  h$. Furthermore, $\mathbbm k \langle c,d, h \rangle/(I+I_1)$ is isomorphic to the Sweedler algebra $H_4$ if $\ch \mathbbm k \ne 2$.
\end{theorem}
\begin{proof}	
Again, it is easy to see that $I+I_1$ is a biideal and $S(I+I_1) \subseteq I+I_1$. Hence $\mathbbm k \langle c,d, h \rangle/(I+I_1)$ is indeed a Hopf algebra. The relations on $\bar c$ and $\bar h$ imply that $\mathbbm k \langle c,d, h \rangle/(I+I_1)$
is the linear span of $\bar 1, \bar c, \bar h, \bar c \bar  h$. In order to prove that $\bar 1, \bar c, \bar h, \bar c \bar  h$ are linearly independent, it is sufficient to provide a unital associative algebra with such relations
where the corresponding elements are indeed linearly independent. 

Let $ \mathbbm k e_1 \oplus \mathbbm k e_2$ be the algebra that is the direct product of two copies of the base field~$ \mathbbm k $, $e_1 e_2 = e_2 e_1 = 0 $, $e_1^2=e_1$, $e_2^2=e_2$. Let $\langle a, b \rangle_{\mathbbm k}$
be the two-dimensional $ \mathbbm k e_1 \oplus \mathbbm k e_2$-bimodule where $e_1 a = a e_2 = a$, $e_2 b = b e_1 = b$, $e_2 a = a e_1 = e_1 b = b e_2 = 0$. Using these structures, define on $ \mathbbm k e_1 \oplus \mathbbm k e_2 \oplus \langle a, b \rangle_{\mathbbm k}$ the structure of an algebra such that $\langle a, b \rangle_{\mathbbm k}^2=0$.
This algebra is associative with the identity element $e_1 + e_2$. Moreover, if we denote $c_1:=a+b+e_1-e_2$, $h_1 := e_2$,
then $c_1$ and $h_1$ satisfy the same relations as $\bar c$ and $\bar h$, which implies that 
$$\mathbbm k \langle c,d, h \rangle/(I+I_1) \cong  \mathbbm k e_1 \oplus \mathbbm k e_2 \oplus \langle a, b \rangle_{\mathbbm k}$$
as algebras, they are both four-dimensional, and $\bar 1, \bar c, \bar h, \bar c \bar  h$ is indeed a basis in $\mathbbm k \langle c,d, h \rangle/(I+I_1)$.

Suppose $\ch \mathbbm k \ne 2$. Then the map $c \mapsto \bar c$ and $v \mapsto \frac{\bar 1 - \bar c}2 + \bar c\bar h$
induces an isomorphism of Hopf algebras $H_4$ and $\mathbbm k \langle c,d, h \rangle/(I+I_1)$.
\end{proof}	

\begin{remark}\label{RemarkTwoPointsCoactUpperTriangularFinDimSelfDual} It is well known that the Sweedler algebra is self-dual if $\ch \mathbbm k \ne 2$ (see e.g. Remark~\ref{RemarkTwoPointsH4H4Star} below). It turns out that the Hopf algebra $\mathbbm k \langle c,d, h \rangle/(I+I_1)$
is self-dual in any characteristic. One of the possible Hopf algebra isomorphisms
 $$\mathbbm k \langle c,d, h \rangle/(I+I_1) \cong \bigl( \mathbbm k \langle c,d, h \rangle/(I+I_1) \bigr)^*$$
 is $$ \bar 1 \mapsto \delta_1 + \delta_c,\quad \bar c \mapsto \delta_1 - \delta_c + \delta_h - \delta_{ch},\quad
 \bar h \mapsto -\delta_h + \delta_c,\quad \bar c \bar h \mapsto - \delta_c$$
 where $\delta_1, \delta_c, \delta_h, \delta_{ch}$ is the basis of $ \bigl( \mathbbm k \langle c,d, h \rangle/(I+I_1) \bigr)^*$ dual to the basis $\bar 1$, $\bar c$, $\bar h$, $\bar c \bar h$ of $\mathbbm k \langle c,d, h \rangle/(I+I_1)$.  
\end{remark}	

\begin{theorem}\label{TheoremTwoPointsCoactClassification} Let $\mathbbm k$ be a field.
Then every unital Hopf algebra coaction on $\mathbbm k \times \mathbbm k$ is equivalent via $\id_{\mathbbm k \times \mathbbm k}$ to one of the following four unital coactions:
\begin{enumerate}
\item \label{ItemTheoremTwoPointsCoactClassification1} $\rho_1 \colon {\mathbbm k \times \mathbbm k} \to (\mathbbm k \times \mathbbm k) \otimes \mathbbm k$, $\rho_1(a)=a\otimes 1$ for all
$a\in \mathbbm k \times \mathbbm k$. The cosupport $\cosupp \rho_1$
is one-dimensional and consists of all scalar operators on $\mathbbm k \times \mathbbm k$.

\item \label{ItemTheoremTwoPointsCoactClassification2} 
 $\rho_2
= \Delta_{(\mathbbm k C_2)^*}$, the comultiplication in the Hopf algebra dual to $\mathbbm k C_2$
where isomorphism $\mathbbm k \times \mathbbm k \cong (\mathbbm k C_2)^*$ is defined by $e=(1,0)\mapsto \delta_1$, $(0,1)\mapsto \delta_c$ and $\delta_1, \delta_c$ is the basis dual to the basis $1,c$ of $\mathbbm k C_2$.
The cosupport $\cosupp \rho_2$
is two-dimensional and consists of all linear operators on $\mathbbm k \times \mathbbm k$
that have matrices from the subalgebra $\left\lbrace \left( \left.\begin{smallmatrix} \alpha & \beta \\
\beta & \alpha \end{smallmatrix}\right) \right| \alpha,\beta \in \mathbbm k   \right\rbrace$
in the basis $(1,0)$, $(0,1)$.

\myitem[3]  \label{ItemTheoremTwoPointsCoactClassification3} $\rho_3 \colon {\mathbbm k \times \mathbbm k} \to (\mathbbm k \times \mathbbm k) \otimes \mathbbm k \langle c,d, h \rangle/I$ where $\rho_3(e):= 1_{\mathbbm k \times \mathbbm k} \otimes \bar h + e\otimes \bar c$ and $\mathbbm k \langle c,d, h \rangle/I$ is the Hopf algebra from Theorem~\ref{TheoremTwoPointsCoactUniversal}.
The cosupport $\cosupp \rho_3$
is three-dimensional and consists of all linear operators on $\mathbbm k \times \mathbbm k$
that have upper triangular matrices in the basis $1_{\mathbbm k \times \mathbbm k},e$. This class of equivalent coactions
contains a representative $\rho_\mathrm{3a}$ with a finite dimensional coacting Hopf algebra, namely, it is sufficient to replace $\mathbbm k \langle c,d, h \rangle/I$ with $\mathbbm k \langle c,d, h \rangle/(I+I_1)$ (see Theorem~\ref{TheoremTwoPointsCoactUpperTriangularFinDim}).
If $\ch \mathbbm k \ne 2$, then this class contains the $H_4$-coaction $\rho_\mathrm{3b} \colon {\mathbbm k \times \mathbbm k} \to (\mathbbm k \times \mathbbm k) \otimes H_4$
where $\rho_\mathrm{3b}(1,-1) = (1,-1) \otimes c + (1,1)\otimes cv$.
\end{enumerate}
Finally, the coactions $\rho_w$ and the corresponding Hopf algebras are $(\cosupp \rho_w)$-universal for $w=1,2,3$.
\end{theorem}
\begin{proof}[Proof of Theorem~\ref{TheoremTwoPointsCoactClassification}]
Consider an arbitrary unital $H$-coaction $\rho \colon \mathbbm k \times \mathbbm k \to (\mathbbm k \times \mathbbm k) \otimes H$
	where $H$ is a Hopf algebra. Define the elements $c$ and $h$ as in Lemma~\ref{LemmaTwoPointsCoactRelations}.
	
	Note that both $1_H$ and $c$ are group-like elements. Therefore they either coincide or are linearly independent.
	(See e.g.~\cite[Proposition~1.4.14]{DNR}.)
	Suppose $c=1_H$. Then, by Lemma~\ref{LemmaTwoPointsCoactRelations}, $hc+ch+c^2 = c$, i.e. $2h=0$.
	If $h=0$, which is always the case if $\ch \mathbbm k \ne 2$, then $\rho(e)=e\otimes 1_H$, i.e.
	$\cosupp \rho$ is the linear span of $\id_{\mathbbm k \times \mathbbm k}$ and we are in the case~\ref{ItemTheoremTwoPointsCoactClassification1}.
	
	When $\ch \mathbbm k = 2$, it is still possible that $c=1_H$, but $h\ne 0$.
	The condition $\varepsilon(h)=0$ implies that in this case $1_H$ and $h$ are linearly independent.
	Applying different linear functions $\langle 1_H, h \rangle_{\mathbbm k} \to \mathbbm k$, which can be extended to elements of $H^*$, we see that $\cosupp \rho$ consists of all linear operators on $\mathbbm k \times \mathbbm k$
	that have matrices from the subalgebra $\left\lbrace \left( \left.\begin{smallmatrix} \alpha & \beta \\
	0 & \alpha \end{smallmatrix}\right) \right| \alpha,\beta \in \mathbbm k   \right\rbrace$
	in the basis $1_{\mathbbm k \times \mathbbm k},e$. Switching to the basis $(1,0)$, $(0,1)$, we obtain that $\cosupp \rho= \cosupp \rho_2$. By Lemma~\ref{LemmaTwoPointsCoactRelations}, the element $h$ satisfies the same relation $h^2=h$ and has the same values of the counit and the comultiplication as $\delta_c$ from the Hopf algebra $(\mathbbm k C_2)^*$ in the case~\ref{ItemTheoremTwoPointsCoactClassification2}.
	
	Now we assume that $1_H$ and $c$ are linearly independent. If $h$ is a linear combination
	of $1_H$ and $c$, then $h$ commutes with $c$, and the relation $hc+ch+c^2 = c$ implies
	that $2h = 1-c$, which is impossible if $\ch \mathbbm k = 2$.
	Suppose $\ch \mathbbm k \ne 2$. Then $h=\frac{1_H-c}2$ and $c^2=1_H$ since $h^2=h$.
	 In the basis $(1,1)=1_{\mathbbm k \times \mathbbm k}$ and $(1,-1)=2e-1_{\mathbbm k \times \mathbbm k}$
	 		we have $\rho(1,-1) = (1,-1)\otimes c$. Applying elements of $H^*$, we see that in this case 
	 		$\cosupp \rho$ consists of all linear operators that have diagonal matrices
	 		in the basis $(1,1)$, $(1,-1)$.  Switching to the basis $(1,0)$, $(0,1)$ and identifying $h$ with $\delta_c$ and $c$ with $(\delta_1 - \delta_c)$, we again
	 		arrive at the case~\ref{ItemTheoremTwoPointsCoactClassification2}.
 
	 Suppose that $1_H$, $c$ and $h$ are linearly independent. Applying elements of $H^*$, we see that we are in the case~\ref{ItemTheoremTwoPointsCoactClassification3}. By Lemma~\ref{LemmaTwoPointsCoactRelations}
	 and Theorem~\ref{TheoremTwoPointsCoactUniversal}, the coaction $\rho_3$ is indeed universal.	
	 
	 	We have shown above that for every characteristic $\ch \mathbbm k$, there exist exactly three distinct possibilities for
	 linear dependencies between $1$, $c$ and $h$, each corresponding to a specific cosupport.
	 In the cases~\ref{ItemTheoremTwoPointsCoactClassification1} and~\ref{ItemTheoremTwoPointsCoactClassification2}, 
	 the coaction factors through $\rho_1$ and $\rho_2$, respectively. Combined with
	 Remark~\ref{RemarkUniversalPropertyOfACoaction}, this implies that the coactions $\rho_1$ and $\rho_2$ are also universal.	 
\end{proof}	

Now we are following the strategy described in Section~\ref{SectionClassificationOutline}:
\begin{theorem}\label{TheoremTwoPointsActClassification} Let $\mathbbm k$ be a field.
	Then every unital Hopf algebra action on $\mathbbm k \times \mathbbm k$ is equivalent to one of the following three unital actions:
	\begin{enumerate}
		\item \label{ItemTheoremTwoPointsActClassification1} $\psi_1 \colon \mathbbm k \otimes  (\mathbbm k \times \mathbbm k) \to \mathbbm k \times \mathbbm k$, $\psi_1(\lambda \otimes a):=\lambda a$ for all
		$\lambda \in \mathbbm k$, $a\in \mathbbm k \times \mathbbm k$. The cosupport $\cosupp \psi_1$
		is one-dimensional and consists of all scalar operators on $\mathbbm k \times \mathbbm k$.
		
		\item  \label{ItemTheoremTwoPointsActClassification2} $\psi_2 \colon    
\mathbbm k C_2 \otimes
  (\mathbbm k \times \mathbbm k)\to  \mathbbm k \times \mathbbm k $
  where $\psi_2\bigl(c \otimes (\alpha, \beta)\bigr) := (\beta,\alpha)$ for all $\alpha, \beta \in
   \mathbbm k $.  
The cosupport $\cosupp \psi_2$
is two-dimensional and consists of all linear operators on $\mathbbm k \times \mathbbm k$
that have matrices from the subalgebra $\left\lbrace \left( \left.\begin{smallmatrix} \alpha & \beta \\
\beta & \alpha \end{smallmatrix}\right) \right| \alpha,\beta \in \mathbbm k   \right\rbrace$
in the basis $(1,0)$, $(0,1)$.
		
		\item  \label{ItemTheoremTwoPointsActClassification3} $\psi_3=
\rho_3^\vee  (\varkappa_{\mathbbm k \langle c,d, h \rangle/I}\otimes \id_ {\mathbbm k \times \mathbbm k})$ where 
$\rho_3 \colon {\mathbbm k \times \mathbbm k} \to (\mathbbm k \times \mathbbm k) \otimes \mathbbm k \langle c,d, h \rangle/I$  is the coaction from
Theorem~\ref{TheoremTwoPointsCoactClassification} and $\mathbbm k \langle c,d, h \rangle/I$ is the Hopf algebra from Theorem~\ref{TheoremTwoPointsCoactUniversal}.
		The cosupport $\cosupp \psi_3$
		is three-dimensional and consists of all linear operators on $\mathbbm k \times \mathbbm k$
		that have upper triangular matrices in the basis $1_{\mathbbm k \times \mathbbm k},e$. If $\ch \mathbbm k \ne 2$, then this class
of equivalent actions
		contains the $H_4$-action $\psi_\mathrm{3b} \colon H_4 \otimes (
\mathbbm k \times \mathbbm k) \to \mathbbm k \times \mathbbm k$
		where $\psi_\mathrm{3b}\bigl(c\otimes(\alpha,\beta)\bigr) := (\beta,\alpha)$,
		$\psi_\mathrm{3b}\bigl(v \otimes (\alpha,\beta)\bigr) := (\alpha-\beta, \alpha-\beta)$ for all $\alpha, \beta \in \mathbbm k$.
	\end{enumerate}
	Finally, the actions $\psi_k$ and the corresponding Hopf algebras are $(\cosupp \psi_k)$-universal for $k=1,2,3$.
\end{theorem}
\begin{proof}
	Consider an arbitrary unital $H$-action $\psi \colon H  \otimes (\mathbbm k \times \mathbbm k) \to \mathbbm k \times \mathbbm k$. Let $V := \cosupp \psi$. Then by Theorem~\ref{TheoremUnivHopf(Co)actDuality}
	the $V$-universal acting Hopf algebra ${}_\square \mathcal H(\mathbbm k \times \mathbbm k,V)$ is isomorphic to $\left(\mathcal H^\square(\mathbbm k \times \mathbbm k,V)\right)^\circ$ where its universal action $\psi^{\mathbf{Hopf}}_{\mathbbm k \times \mathbbm k,V}$
	coincides with $\bigl(\rho^\mathbf{Hopf}_{\mathbbm k \times \mathbbm k,V}\bigr)^\vee(\varkappa_{\mathcal{H}^\square(\mathbbm k \times \mathbbm k,V)} \otimes \id_{\mathbbm k \times \mathbbm k})$.
	Recall that for the action $\psi$ we have $\cosupp \psi = V$, which implies that 
	$\cosupp \psi^{\mathbf{Hopf}}_{\mathbbm k \times \mathbbm k,V} = \cosupp \rho^\mathbf{Hopf}_{\mathbbm k \times \mathbbm k,V} = V$.
	Therefore, $V$ must coincide with one of the possible cosupports listed in Theorem~\ref{TheoremTwoPointsCoactClassification}.
	At the same time, each equivalence class of coactions in Theorem~\ref{TheoremTwoPointsCoactClassification} contains a representative
	with a finite dimensional Hopf algebra. Hence for $\mathbbm k \times \mathbbm k$ the cosupport of the universal action dual to the universal coaction cannot become smaller than the cosupport $V$ of the coaction and for every $\ch \mathbbm k$ we still have three cases. In the cases~\ref{ItemTheoremTwoPointsActClassification1} and~\ref{ItemTheoremTwoPointsActClassification2} we calculate the dual of the $V$-universal coacting Hopf algebra and its coaction.
	
	An explicit verification shows that $\psi_\mathrm{3b} \colon H_4 \otimes (
	\mathbbm k \times \mathbbm k) \to \mathbbm k \times \mathbbm k$ is indeed an action	and, if $\ch \mathbbm k \ne 2$, for each basis in $ \mathbbm k \times \mathbbm k$
	where the identity element $(1,1)$ is the first basis vector, $\cosupp \psi_\mathrm{3b}$ consists of all the operators that have upper triangular matrices.
\end{proof}
\begin{remark}
	(Co)actions on $ \mathbbm k \times \mathbbm k$ were classified in Theorems~\ref{TheoremTwoPointsCoactClassification} and~\ref{TheoremTwoPointsActClassification} up to an equivalence via $\id_{k \times \mathbbm k}$. However,
	$\Aut( \mathbbm k \times \mathbbm k) \cong C_2$ where $c(\alpha,\beta)=(\beta, \alpha)$ for all $\alpha,\beta\in\mathbbm k$.
	All the cosupports in Theorems~\ref{TheoremTwoPointsCoactClassification} and~\ref{TheoremTwoPointsActClassification} are stable under this $\Aut( \mathbbm k \times \mathbbm k)$-action, which implies that we actually have the classification up to an arbitrary equivalence.
\end{remark}
\begin{remark} If we identify $\mathbbm k \langle c,d, h \rangle/(I+I_1)$
	and $\bigl( \mathbbm k \langle c,d, h \rangle/(I+I_1) \bigr)^*$ by the isomorphism from Remark~\ref{RemarkTwoPointsCoactUpperTriangularFinDimSelfDual},
	the unital action $\rho_\mathrm{3a}^\vee \colon \mathbbm k \langle c,d, h \rangle/(I+I_1) \otimes (
	\mathbbm k \times \mathbbm k) \to \mathbbm k \times \mathbbm k$ is uniquely determined by the equalities 
	$\rho_\mathrm{3a}^\vee (c \otimes e) = 1_{\mathbbm k \times \mathbbm k}- e$
	and 	$\rho_\mathrm{3a}^\vee (h \otimes e) = e-1_{\mathbbm k \times \mathbbm k}$.
\end{remark}
\begin{remark}\label{RemarkTwoPointsH4H4Star}
	Suppose $\ch \mathbbm k \ne 2$. Define the automorphism $\varphi \colon H_4 \mathrel{\widetilde{\to}} H_4$
	and the isomorphism $\xi \colon H_4 \mathrel{\widetilde{\to}} H_4^*$
	 of Hopf algebras by $$\varphi(1):=1,\ \varphi(c):=c,\ \varphi(v):=2v,\ \varphi(cv):=2cv,$$ 
	 $$\xi(1) := \delta_1+\delta_c,\ \xi(c) := \delta_1 - \delta_c,\ \xi(v):= \delta_{cv}-\delta_v,\  
	 \xi(cv) := \delta_v+\delta_{cv}.$$
	 Then $\rho_\mathrm{3b}^\vee$ and $\psi_\mathrm{3b}$ can be identified via the composition $\xi\varphi$. 
\end{remark}

\section{Dual numbers}\label{SectionDualNumbers}

Fix the basis $\bar 1 := 1+ (x^2)$ and  $\bar x := x + (x^2)$ in the algebra $ \mathbbm k[x]/(x^2)$ of dual numbers
over a field $\mathbbm k$.

\begin{lemma}\label{LemmaDualNumbersCoactRelations} Let  $\rho \colon \mathbbm k[x]/(x^2) \to \mathbbm k[x]/(x^2) \otimes H$ be an arbitrary unital $H$-coaction where $H$ is a Hopf algebra. Define $c,h\in H$ by $\rho(\bar x) =  \bar 1 \otimes h + \bar x \otimes c$.
	Then  $c$ is invertible, $$h^2=0,\quad hc+ch=0,$$   $$\Delta c= c\otimes c,\quad \Delta h = h \otimes c + 1_H \otimes h,$$
	$$\varepsilon(c)=1,\quad \varepsilon(h) = 0,$$ $$Sc = c^{-1},\quad Sh = -hc^{-1}.$$
\end{lemma}
\begin{proof}
	The proof is completely analogous to the proof of Lemma~\ref{LemmaTwoPointsCoactRelations}.
\end{proof}	

\begin{theorem}\label{TheoremDualNumbersCoactUniversal} Consider the free unital associative algebra $\mathbbm k \langle c,d, h \rangle$ over a field $\mathbbm k$ with
	free generators $c,d, h$, i.e. the algebra of polynomials in the non-commuting variables $c,d, h$.
	Denote by $I_0$ the ideal of $\mathbbm k \langle c,d, h \rangle$ generated by the elements
	$$cd-1,\quad dc-1,\quad h^2,\quad hc+ch.$$
	Define $\Delta \colon \mathbbm k \langle c,d, h \rangle \to  \mathbbm k \langle c,d, h \rangle \otimes  \mathbbm k \langle c,d, h \rangle$ and $\varepsilon \colon \mathbbm k \langle c,d, h \rangle \to  \mathbbm k$
	as the unital algebra homomorphisms such that $$\Delta c= c\otimes c,\quad \Delta d= d\otimes d,\quad \Delta h = h \otimes c + 1 \otimes h,$$
	$$\varepsilon(c)=\varepsilon(d)=1,\quad \varepsilon(h) = 0.$$
	Define $S \colon \mathbbm k \langle c,d, h \rangle \to \mathbbm k \langle c,d, h \rangle$ 
	as the unital algebra anti-homomorphism such that $$Sc = d,\ Sd = c,\ Sh = -hd.$$
	Then $$S I_0 \subseteq I_0,\quad \varepsilon(I_0)=0,\quad \Delta I_0\subseteq
	I_0 \otimes \mathbbm k \langle c,d, h \rangle +
	\mathbbm k \langle c,d, h \rangle \otimes I_0$$ and the algebra $\mathbbm k \langle c,d, h \rangle/I_0$ with the induced maps $\Delta, \varepsilon, S$ is a Hopf algebra. Finally, $\mathbbm k \langle c,d, h \rangle/I_0=\mathcal H^\square\bigl(\mathbbm k[x]/(x^2)\bigr)$, the Manin universal coacting Hopf algebra  on $\mathbbm k[x]/(x^2)$.
\end{theorem}
\begin{proof} In Theorem~\ref{TheoremTwoPointsCoactUniversal} above we have noticed
	that $\mathbbm k \langle c,d, h \rangle$ is a bialgebra and shown
	that $$\Delta(cd-1), \Delta(dc-1) \in I_0 \otimes \mathbbm k \langle c,d, h \rangle +
	\mathbbm k \langle c,d, h \rangle \otimes I_0.$$
	
	Note that
	$$\Delta(h^2)=h^2 \otimes c^2 + h\otimes( ch + hc) + 1\otimes h^2 \in I \otimes \mathbbm k \langle c,d, h \rangle +
			\mathbbm k \langle c,d, h \rangle \otimes I,$$
	\begin{equation*}\begin{split}\Delta(hc+ch)=hc \otimes c^2 + c\otimes hc  + ch \otimes c^2 + c\otimes ch
			\\= (hc+ch)\otimes c^2 + c\otimes (hc+ch) \in I_0 \otimes \mathbbm k \langle c,d, h \rangle +
			\mathbbm k \langle c,d, h \rangle \otimes I_0.\end{split}\end{equation*}
	Hence $I$ is a biideal and  $\mathbbm k \langle c,d, h \rangle/I_0$ is a bialgebra.
	
	Again, $S(cd-1)=dc-1 \in I_0$, $S(dc-1)=cd-1 \in I_0$.
	Moreover, $$dh+hd=d(hc+ch)d+dh(1-cd)+(1-dc)hd \in I_0.$$
	Hence
	$$S(h^2)=hdhd = h(dh+hd)d-h^2d^2 \in I_0.$$
	Finally,
	$$S(hc+ch)=-dhd-hd^2 = -(dh+hd)d \in I_0.$$
	Therefore, $SI_0 \subseteq I_0$ and $S$ is well defined on $\mathbbm k \langle c,d, h \rangle/I_0$.
	An explicit verification shows that $S$ is indeed the antipode and $\mathbbm k \langle c,d, h \rangle/I_0$ is a Hopf algebra.
	
	Define the unital homomorphism $\rho \colon  \mathbbm k[x]/(x^2) \to  \mathbbm k[x]/(x^2) \otimes \mathbbm k \langle c,d, h \rangle/I_0$ such that $\rho(\bar x):= \bar 1 \otimes \bar h + \bar x\otimes \bar c$
	where $\bar a := a + I_0$ for all $a\in \mathbbm k \langle c,d, h \rangle$.
	Then $\rho$ is a unital coaction. Lemma~\ref{LemmaDualNumbersCoactRelations} implies that $\rho$ is universal among all unital coactions on $ \mathbbm k[x]/(x^2)$.
\end{proof}	

\begin{theorem}\label{TheoremDualNumbersCoactClassification} Let $\mathbbm k$ be a field.
	Then every unital Hopf algebra coaction on $\mathbbm k[x]/(x^2)$ is equivalent via $\id_{\mathbbm k[x]/(x^2)}$ to one of the following five unital coactions:
	\begin{enumerate}
		\item \label{ItemTheoremDualNumbersCoactClassification1} $\rho_1 \colon {\mathbbm k[x]/(x^2)} \to \mathbbm k[x]/(x^2) \otimes \mathbbm k$, $\rho_1(a)=a\otimes 1$ for all
		$a\in \mathbbm k[x]/(x^2)$. The cosupport $\cosupp \rho_1$
		is one-dimensional and consists of all scalar operators on $\mathbbm k[x]/(x^2)$.
		
		\item  \label{ItemTheoremDualNumbersCoactClassification2} (if $\ch \mathbbm k = 2$) $\rho_2 = \Delta_{\mathbbm k[x]/(x^2)}$
		where the structure of a Hopf algebra on $\mathbbm k[x]/(x^2)$ is defined by 
		$\Delta \bar x :=\bar x \otimes \bar 1 + \bar 1 \otimes \bar x$, $\varepsilon(\bar x):=0$, $S\bar x=\bar x$.  The cosupport $\cosupp \rho_2$
		is two-dimensional and consists of all linear operators on $\mathbbm k[x]/(x^2)$
		that have matrices from the subalgebra $\left\lbrace \left( \left.\begin{smallmatrix} \alpha & \beta \\
		0 & \alpha \end{smallmatrix}\right) \right| \alpha,\beta \in \mathbbm k   \right\rbrace$
		in the basis $\bar 1,\bar x$.

		\item \label{ItemTheoremDualNumbersCoactClassification3} (if $\ch \mathbbm k = 2$)
		$\rho_{3,\theta}
		= \Delta_{\mathbbm k C_2}$, the comultiplication in the group Hopf algebra $\mathbbm k C_2$
		where $C_2=\langle c \rangle$ is the cyclic group of order $2$ and the isomorphism $\mathbbm k[x]/(x^2) \cong \mathbbm k C_2$ is defined by $\bar 1\mapsto 1$, $\bar x \mapsto \theta(1+c)$,
		$\theta \in  \mathbbm k \backslash \lbrace 0 \rbrace$ is a fixed element. In other words,
		$\rho_{3,\theta}$ corresponds to the standard $C_2$-grading on $\mathbbm k C_2$.
		The cosupport $\cosupp \rho_{3,\theta}$
is two-dimensional and consists of all linear operators on $\mathbbm k[x]/(x^2)$
that have matrices from the subalgebra $\left\lbrace \left( \left.\begin{smallmatrix} \alpha & \theta(\alpha+\beta) \\
0 & \beta \end{smallmatrix}\right) \right| \alpha,\beta \in \mathbbm k   \right\rbrace$ in the basis $\bar 1,\bar x$.

	\item \label{ItemTheoremDualNumbersCoactClassification4}
	 $\rho_4 \colon {\mathbbm k[x]/(x^2)} \to \mathbbm k[x]/(x^2) \otimes \mathbbm k C_\infty$ where $\rho_4(\bar x)=\bar x\otimes c$ where $C_\infty=\langle c \rangle$ is the infinite cyclic group and $ \mathbbm k C_\infty \cong \mathbbm k[c,c^{-1}]$,
	 the algebra of Laurent polynomials, i.e. $\rho_4$ corresponds to the $C_\infty$-grading on $\mathbbm k[x]/(x^2)$
	 where $\bar x \in \bigl(\mathbbm k[x]/(x^2)\bigr)^{(c)}$.	
		The cosupport $\cosupp \rho_4$
		is two-dimensional and consists of all linear operators on $\mathbbm k[x]/(x^2)$
		that have diagonal matrices in the basis $\bar 1,\bar x$. This class of equivalent coactions
		contains a representative with a finite dimensional coacting Hopf algebra, namely, the $\mathbbm k C_2$-coaction $\rho_\mathrm{4a} \colon {\mathbbm k[x]/(x^2)} \to \mathbbm k[x]/(x^2) \otimes \mathbbm k C_2$
		where $\rho_\mathrm{4a}(\bar x) = \bar x \otimes c$.

		\item  \label{ItemTheoremDualNumbersCoactClassification5} $\rho_5 \colon {\mathbbm k[x]/(x^2)} \to \mathbbm k[x]/(x^2) \otimes \mathbbm k \langle c,d, h \rangle/I_0$ where $\rho_5(\bar x):= \bar 1 \otimes \bar h + \bar x\otimes \bar c$ and $\mathbbm k \langle c,d, h \rangle/I_0$ is the Hopf algebra from Theorem~\ref{TheoremDualNumbersCoactUniversal}.
		The cosupport $\cosupp \rho_5$
		is three-dimensional and consists of all linear operators on $\mathbbm k[x]/(x^2)$
		that have upper triangular matrices in the basis $\bar 1,\bar x$. This class of equivalent coactions
		contains a representative with a finite dimensional coacting Hopf algebra, namely, the $H_4$-coaction $\rho_\mathrm{5a} \colon {\mathbbm k[x]/(x^2)} \to \mathbbm k[x]/(x^2) \otimes H_4$
		where $\rho_\mathrm{5a}(\bar x) = \bar x \otimes c + \bar 1\otimes cv$.
	\end{enumerate}
	Finally, the coactions $\rho_1, \rho_2, \rho_{3,\theta}, \rho_4,\rho_5$ and the corresponding Hopf algebras are $V$-universal for their cosupports $V$.
\end{theorem}
\begin{remark}
	In the case $\ch \mathbbm k \ne 2$ the classification of the coactions above first appeared in an (unpublished) undergraduate project of Vladimir Nikitenko. For $\ch \mathbbm k = 2$ the automorphism group scheme of $\mathbbm k[x]/(x^2)$
	is considered e.g. in~\cite[Chapter 7, Exercise 16]{Waterhouse}.
\end{remark}
\begin{proof}[Proof of Theorem~\ref{TheoremDualNumbersCoactClassification}]
	Consider an arbitrary unital $H$-coaction $\rho \colon \mathbbm k[x]/(x^2) \to \mathbbm k[x]/(x^2) \otimes H$
	where $H$ is a Hopf algebra. Define the elements $c$ and $h$ as in Lemma~\ref{LemmaDualNumbersCoactRelations}.
	
	Again, both $1_H$ and $c$ are group-like elements. Therefore they either coincide or are linearly independent.
	Suppose $c=1_H$. Then $2h=0$ since, by Lemma~\ref{LemmaDualNumbersCoactRelations}, $hc+ch=0$. 
	If $h=0$, which is always the case if $\ch \mathbbm k \ne 2$, then $\rho(\bar x)=\bar x\otimes 1_H$, i.e.
	$\cosupp \rho$ is the linear span of $\id_{\mathbbm k[x]/(x^2)}$ and we are in the case~\ref{ItemTheoremDualNumbersCoactClassification1}. The coaction $\rho_1$ is obviously universal.
	
	When $\ch \mathbbm k = 2$, it is still possible that $c=1_H$, but $h\ne 0$.
	The condition $\varepsilon(h)=0$ implies that in this case $1_H$ and $h$ are linearly independent.
	Applying different linear functions $\langle 1_H, h \rangle_{\mathbbm k} \to \mathbbm k$, which can be extended to elements of $H^*$, we see that $\cosupp \rho = \cosupp \rho_2$. By Lemma~\ref{LemmaDualNumbersCoactRelations} the element $h$ satisfies the same relation $h^2=0$ and has the same values of the counit and the comultiplication as $\bar x$ from the Hopf algebra structure on $\mathbbm k[x]/(x^2)$ in the case~\ref{ItemTheoremDualNumbersCoactClassification2}. By Remark~\ref{RemarkUniversalPropertyOfACoaction}, the coaction $\rho_2$ is universal.
	
	Now we assume that $1_H$ and $c$ are linearly independent.
	 Suppose $h=\alpha 1_H + \beta c$ for some $\alpha, \beta \in\mathbbm k$. Then $h$ commutes with $c$, and the relation $hc+ch=0$ implies that $2h = 0$.
	 
	 Suppose $h=0$, which is always the case if $\ch \mathbbm k \ne 2$.
	 	Applying elements of $H^*$, we see that in this case 
	$\cosupp \rho$ consists of all linear operators that have diagonal matrices
	in the basis $\bar 1$, $\bar x$, i.e. we are in the case~\ref{ItemTheoremDualNumbersCoactClassification4}.
	The coaction $\rho_4$ is universal by Remark~\ref{RemarkUniversalPropertyOfACoaction}, since the linear span of all powers of $c$ in $H$ 	is a homomorphic image of $\mathbbm k C_\infty$.
	
	Suppose $\ch \mathbbm k = 2$ and $h\ne 0$. From the relation $h^2=0$ we obtain that $\alpha^2 1_H + \beta^2 c^2 = 0$.
	Hence $c^2=1$, $\alpha^2 = \beta^2$, which in the case of $\ch \mathbbm k = 2$ implies that $\alpha =\beta$.
	Let $\theta := \alpha = \beta$. Then $\rho(\bar x) = \bar 1 \otimes \theta(1+c) + \bar x \otimes c$.
	Applying elements of $H^*$, we see that in this case 
	$\cosupp \rho = \cosupp \rho_{3,\theta}$. The coaction $\rho_{3,\theta}$ is universal by Remark~\ref{RemarkUniversalPropertyOfACoaction}, since the linear span of $1_H$ and $c$ in $H$ is a Hopf algebra isomorphic to $\mathbbm k C_2$.
	
	Suppose that $1_H$, $c$ and $h$ are linearly independent. Applying elements of $H^*$, we see that we are in the case~\ref{ItemTheoremDualNumbersCoactClassification5}. Lemma~\ref{LemmaDualNumbersCoactRelations} and Theorem~\ref{TheoremDualNumbersCoactUniversal} imply that the coaction $\rho_5$ is indeed universal.	
\end{proof}	

The proof of Theorem~\ref{TheoremDualNumbersActClassification} below is completely analogous to the proof of Theorem~\ref{TheoremTwoPointsActClassification} above:
\begin{theorem}\label{TheoremDualNumbersActClassification} Let $\mathbbm k$ be a field.
	Then every unital Hopf algebra action on $\mathbbm k[x]/(x^2)$ is equivalent to one of the following five unital actions:
	\begin{enumerate}
		\item \label{ItemTheoremDualNumbersActClassification1} $\psi_1 \colon \mathbbm k \otimes  {\mathbbm k[x]/(x^2)} \to \mathbbm k[x]/(x^2) $, $\psi_1(\lambda \otimes a)=\lambda a$ for all
		$\lambda \in \mathbbm k$, $a\in \mathbbm k[x]/(x^2)$. The cosupport $\cosupp \psi_1$
		is one-dimensional and consists of all scalar operators on $\mathbbm k[x]/(x^2)$.
		
		\item  \label{ItemTheoremDualNumbersActClassification2} (if $\ch \mathbbm k = 2$) $\psi_2 \colon    
		\mathbbm k [x]/(x^2) \otimes
		\mathbbm k[x]/(x^2)\to  \mathbbm k[x]/(x^2) $,
		where the structure of a Hopf algebra on $\mathbbm k[x]/(x^2)$ was defined in the case~\ref{ItemTheoremDualNumbersCoactClassification2} of Theorem~\ref{TheoremDualNumbersCoactClassification},
		and
		 $\psi_2(\bar x \otimes \bar x) := \bar 1$.  
		The cosupport $\cosupp \psi_2$
		consists of all linear operators on $\mathbbm k[x]/(x^2)$
		that have matrices from the subalgebra $\left\lbrace \left( \left.\begin{smallmatrix} \alpha & \beta \\
		0 & \alpha \end{smallmatrix}\right) \right| \alpha,\beta \in \mathbbm k   \right\rbrace$
		in the basis $\bar 1,\bar x$.

		\item  \label{ItemTheoremDualNumbersActClassification3} (if $\ch \mathbbm k = 2$) $\psi_{3,\theta} \colon    
(\mathbbm k\times \mathbbm k) \otimes
\mathbbm k[x]/(x^2)\to  \mathbbm k[x]/(x^2)$,
where the  Hopf algebra isomorphism $\mathbbm k\times \mathbbm k \cong (\mathbbm k C_2)^*$ was defined in the case~\ref{ItemTheoremTwoPointsCoactClassification2} of Theorem~\ref{TheoremTwoPointsCoactClassification},
and
$$\psi_{3,\theta}\bigl((\alpha,\beta) \otimes \bar 1 \bigr) := \beta\bar 1,\quad
\psi_{3,\theta}\bigl((\alpha,\beta) \otimes \bar x\bigr) := \theta(\alpha+\beta) \bar 1 + \alpha \bar x\text{ for all } \alpha, \beta \in
\mathbbm k ,$$
$\theta \in  \mathbbm k \backslash \lbrace 0 \rbrace$ is a fixed element.  
The cosupport $\cosupp \psi_{3,\theta}$
consists of all linear operators on $\mathbbm k[x]/(x^2)$
that have matrices from the subalgebra $\left\lbrace \left( \left.\begin{smallmatrix} \alpha & \theta(\alpha+\beta) \\
0 & \beta \end{smallmatrix}\right) \right| \alpha,\beta \in \mathbbm k   \right\rbrace$
in the basis $\bar 1,\bar x$.

	\item \label{ItemTheoremDualNumbersActClassification4}
$\psi_4 \colon (\mathbbm kC_\infty)^\circ \otimes {\mathbbm k[x]/(x^2)} \to \mathbbm k[x]/(x^2)$
 where $\psi_4(\lambda \otimes \bar x):=\lambda(c)\bar x$ for all $\lambda \in (\mathbbm kC_\infty)^\circ$.	
The cosupport $\cosupp \psi_4$
is two-dimensional and consists of all linear operators on $\mathbbm k[x]/(x^2)$
that have diagonal matrices in the basis $\bar 1,\bar x$. If $\ch \mathbbm k \ne 2$, then this class
of equivalent actions
contains the $\mathbbm k C_2$-action $\psi_\mathrm{4a} \colon \mathbbm kC_2 \otimes {\mathbbm k[x]/(x^2)} \to \mathbbm k[x]/(x^2)$
where $\psi_\mathrm{4a}(c \otimes \bar x) := -\bar x$.

		\item  \label{ItemTheoremDualNumbersActClassification5} $\psi_5=
		\rho_5^\vee  (\varkappa_{\mathbbm k \langle c,d, h \rangle/I_0}\otimes \id_ {\mathbbm k[x]/(x^2)})$ where 
		$\rho_5 \colon {\mathbbm k[x]/(x^2)} \to (\mathbbm k[x]/(x^2)) \otimes \mathbbm k \langle c,d, h \rangle/I$  is the coaction from
		Theorem~\ref{TheoremDualNumbersCoactClassification} and $\mathbbm k \langle c,d, h \rangle/I$ is the Hopf algebra from Theorem~\ref{TheoremDualNumbersCoactUniversal}.
		The cosupport $\cosupp \psi_5$
		is three-dimensional and consists of all linear operators on $\mathbbm k[x]/(x^2)$
		that have upper triangular matrices in the basis $\bar 1,\bar x$. If $\ch \mathbbm k \ne 2$, then this class
		of equivalent actions
		contains the $H_4$-action $\psi_\mathrm{5a} \colon H_4 \otimes 
		\mathbbm k[x]/(x^2) \to \mathbbm k[x]/(x^2)$
		where $\psi_\mathrm{5a}(c \otimes \bar x) := -\bar x$,
		$\psi_\mathrm{5a}(v \otimes \bar x) := \bar 1$. (Note that  $\psi_\mathrm{5a}= \rho_\mathrm{5a}^{\vee}$
		 if we identify $H_4$ and $H_4^*$ via the isomorphism $\xi$ from Remark~\ref{RemarkTwoPointsH4H4Star}.)
	\end{enumerate}
	Finally, the actions $\psi_1, \psi_2, \psi_{3,\theta}, \psi_4,\psi_5$ and the corresponding Hopf algebras are $V$-universal for their cosupports $V$.
\end{theorem}
\begin{remark}
	In the case $\ch \mathbbm k \ne 2$ the actions above were first classified in~\cite[Theorem~6.7]{ASGordienko20ALAgoreJVercruysse}.
\end{remark}
\begin{remark}
	Recall that if we identify $\lambda \in \mathbbm k[c,c^{-1}]^*$
	with the sequence $\lambda_n := \lambda\left( c^n \right)$,
	where $n\in\mathbb Z$, then $\lambda \in \mathbbm k[c,c^{-1}]^\circ$
	if and only if the sequence $\lambda$ satisfies some linear recurrence relation with constant coefficients. Denote by $\mathbbm k^\times$ the group of invertible elements of the field $\mathbbm k$.
	The sequences $(n^k \gamma^n)_n$, where $k\in \mathbb Z_+$, $\gamma\in \mathbbm k^\times$,
	form a basis in $\mathbbm k[c,c^{-1}]^\circ$. Note that $(n)_n$ is a primitive element
	and $(\mu^n)_n$ are group-like elements of $\mathbbm k[c,c^{-1}]^\circ$.
	Hence we have a Hopf algebra isomorphism	
	  $(\mathbbm kC_\infty)^\circ = \mathbbm k[c,c^{-1}]^\circ \cong \mathbbm k[t]\otimes \mathbbm k\mathbbm k^\times$
	where $\Delta t = t\otimes 1 + 1 \otimes t$, 
	$ (n^k \gamma^n)_n \mapsto t^k \otimes \gamma$
	 for all $\gamma \in \mathbbm k^\times$, $k\in\mathbb Z_+$. Under this isomorphism, $\psi_4\bigl((t^k \otimes \gamma)\otimes \bar x\bigr) = \gamma \bar x$.
\end{remark}
\begin{remark}
	(Co)actions on $\mathbbm k[x]/(x^2)$ are classified in Theorems~\ref{TheoremDualNumbersCoactClassification} and~\ref{TheoremDualNumbersActClassification} up to an equivalence via $\id_{\mathbbm k[x]/(x^2)}$.
	Note that $\Aut(\mathbbm k[x]/(x^2)) \cong \mathbbm k^\times$ where $\gamma \cdot \bar 1 :=\bar 1 $ and $\gamma \cdot \bar x :=\gamma \bar x $ for all $\gamma \in\mathbbm k^\times$.
	All the cosupports in Theorems~\ref{TheoremDualNumbersCoactClassification} and~\ref{TheoremDualNumbersActClassification} are stable under this $\Aut( \mathbbm k[x]/(x^2))$-action
	except the ones in the case~\ref{ItemTheoremDualNumbersCoactClassification3},
	which can be identified for different~$\theta$ by the $\Aut( \mathbbm k[x]/(x^2))$-action.
	This implies that, up to an equivalence via an arbitrary isomorphism, we have exactly five different classes of equivalent coactions and five ones of actions. They are all listed
	in Theorems~\ref{TheoremDualNumbersCoactClassification} and~\ref{TheoremDualNumbersActClassification}.
\end{remark}

We conclude this section by proving that the Manin universal coacting Hopf algebras of the set of two points and of the dual numbers are different.

\begin{theorem}\label{TheoremManinForTwoPointsAndDualNumbersAreDifferent}
	The Hopf algebras $\mathbbm k\langle c,d,h \rangle/I$ from Theorem~\ref{TheoremTwoPointsCoactUniversal}
	and $\mathbbm k\langle c,d,h \rangle/I_0$ from Theorem~\ref{TheoremDualNumbersCoactUniversal}
	are not isomorphic for any field $\mathbbm k$.
\end{theorem}
\begin{proof}
	Suppose $\mathbbm k\langle c,d,h \rangle/I \cong \mathbbm k\langle c,d,h \rangle/I_0$.
	In the case~\ref{ItemTheoremDualNumbersCoactClassification4} of Theorem~\ref{TheoremDualNumbersCoactClassification}
	we have shown that the Hopf algebra $\mathbbm k C_\infty \cong \mathbbm k[t,t^{-1}]$ is a homomorphic image of $\mathbbm k\langle c,d,h \rangle/I_0$. Therefore, there exists a surjective Hopf algebra homomorphism
	$\varphi \colon \mathbbm k\langle c,d,h \rangle/I \to \mathbbm k[t,t^{-1}]$.
	Recall that different group-like elements are linearly independent, which implies that $\lbrace t^n \mid n\in\mathbb Z \rbrace$ is the set of  group-like elements in $\mathbbm k[t,t^{-1}]$.  Being a coalgebra homomorphism, $\varphi$
	maps group-like elements to group-like elements. Thus $\varphi(\bar c) = t^k$
	for some $k\in\mathbb Z$. The algebra $\mathbbm k[t,t^{-1}]$ is commutative, hence
	the relation $\bar h \bar c + \bar c \bar h + \bar c^2  = \bar c$
	implies that $2\varphi(\bar h) = \varphi(\bar 1) - \varphi(\bar c)$.
	
	If $\ch \mathbbm k \ne 2$, then $\varphi(\bar h) = \frac{1}{2}(\varphi(\bar 1) - \varphi(\bar c))$,
	and from relation $\bar h^2 = \bar h$ it follows that $\varphi(\bar c)^2 = \varphi(\bar 1)$,
	i.e. $t^{2k}=1$. Hence $k=0$, $\varphi(\bar h)=0$, $\varphi(\bar c)=1$ and the image of $\varphi$ is one-dimensional, which contradicts the surjectivity of $\varphi$.
	
	If $\ch \mathbbm k = 2$, then $\varphi(\bar c) = \varphi(\bar 1)=1$, i.e. $\varphi(\bar h)$ is a primitive element
	of the group algebra $\mathbbm k[t,t^{-1}]$. An explicit check shows that in every group algebra there are no nonzero primitive elements. Hence again $\varphi(\bar h)=0$, $\varphi(\bar c)=1$, which contradicts the surjectivity of $\varphi$.
	
	 Therefore, $\mathbbm k\langle c,d,h \rangle/I \ncong \mathbbm k\langle c,d,h \rangle/I_0$.
\end{proof}	 
\begin{remark}
	As pointed out by the referee, an alternative way to prove Theorem~\ref{TheoremManinForTwoPointsAndDualNumbersAreDifferent} is to quotient both Hopf algebras by their ideals generated by commutators. This yields the Hopf algebras corresponding to the automorphism group schemes of $\mathbbm k \times \mathbbm k$
	and $\mathbbm k[x]/(x^2)$, respectively, which are not isomorphic.
\end{remark}

\section{Straight line}

Recall that $\mathbbm k [x]$ is the algebra of regular functions on the straight line over a field $\mathbbm k$.

\begin{theorem}\label{TheoremYuIManinHopfAlgebraStraightLineDoesNotExist} The (ungraded) Manin universal coacting Hopf algebra
	$\mathcal H^\square(\mathbbm k [x])$ on $\mathbbm k [x]$ does not exist for any field $\mathbbm k$.
\end{theorem}
\begin{proof}
	Suppose that $\mathcal H^\square(\mathbbm k [x])$ does exist.
	Let $$\rho^{\mathbf{Hopf}}_{\mathbbm k [x]} \colon \mathbbm k [x] \to \mathbbm k [x] \otimes \mathcal H^\square(\mathbbm k [x])$$ be the corresponding
	universal coaction. Then $\rho^{\mathbf{Hopf}}_{\mathbbm k [x]}(x) = \sum\limits_{k=0}^n x^k \otimes h_k$
	for some $n\in\mathbb Z_+$, $h_k \in \mathcal H^\square(\mathbbm k [x])$.
	
	Choose any $m \in\mathbb N$ such that $2m > n$ and consider
	the unital algebra homomorphism $\rho_m \colon \mathbbm k [x] \to \mathbbm k [x] \otimes H_4$
	defined by $\rho_m (x):= x \otimes c + x^{2m} \otimes cv$.
	
	We claim that $\rho_m$ is an $H_4$-coaction. Since $\rho_m$, $\varepsilon$ and $\Delta$ are unital algebra homomorphisms,
	it is sufficient to verify the comodule axioms on $x$. First of all, note that $\rho_m \left(x^2\right) = x^2 \otimes 1$.
	Hence
	 $$(\rho_m \otimes \id_{H_4})\rho_m = x \otimes c \otimes c + x^{2m} \otimes cv \otimes c + x^{2m} \otimes 1 \otimes cv
	 = (\id_{\mathbbm k [x]}  \otimes \Delta)\rho_m.$$
	 In addition, the counit axiom holds trivially. Hence		
	 $\rho_m$ is indeed a coaction.
	 
	At the same time, there exists no Hopf algebra homomorphism $\varphi \colon 
	\mathcal H^\square(\mathbbm k [x]) \to H_4$ such that
	$\rho_m = (\id_{\mathbbm k [x]} \otimes \varphi)\rho^{\mathbf{Hopf}}_{\mathbbm k [x]}$
	since $2m > n$. We get a contradiction. Therefore $\mathcal H^\square(\mathbbm k [x])$ does not exist.
\end{proof}	

\section{The cosupport may reduce under the dualization}\label{SectionCosuppReduces}

Given a Hopf algebra coaction $\rho \colon A \to A \otimes H$, the dual action
$\rho^\vee(\varkappa_{H} \otimes \id_A)
\colon H^\circ \otimes A \to A$
 cannot have a larger cosupport. Below we prove necessary and sufficient conditions for a coaction resulting from a finite grading and the dual action to have the same cosupport.

\begin{theorem}
	Let $\Gamma \colon A=\bigoplus_{g \in G} A^{(g)}$ be a grading by a group $G$ on an algebra $A$ over a field $\mathbbm k$
	such that the support $\supp \Gamma$ is finite. Let $\rho \colon A \to A \otimes \mathbbm k G$ be the corresponding $\mathbbm k G$-coaction, i.e. $\rho(a):=a\otimes g$ for all $a \in A^{(g)}$, $g\in G$. 
	Then $\cosupp \rho = \cosupp \bigl(\rho^\vee(\varkappa_{\mathbbm k G} \otimes \id_A)\bigr)$
	if and only if there exists an ideal $I\subseteq \mathbbm k G$, $\dim(\mathbbm k G/I) < +\infty$,
	such that the images of the elements of $\supp \Gamma$ are linearly independent in $\mathbbm k G/I$.	
\end{theorem}
\begin{proof}
	The algebra $\cosupp \rho$ consists of all $\mathbbm k$-linear operators $A\to A$ that are scalar on each $A^{(g)}$,
	while $\cosupp \bigl(\rho^\vee(\varkappa_{\mathbbm k G} \otimes \id_A)\bigr)$ is the subalgebra of 
	$\cosupp \rho$ that consists of all such operators that correspond to $\lambda \in A^{\circ}$.
	Let $\supp \Gamma = \lbrace g_1, \ldots, g_s \rbrace$.
	
	Suppose $\cosupp \rho = \cosupp \bigl(\rho^\vee(\varkappa_{\mathbbm k G} \otimes \id_A)\bigr)$.
	Then there exists $\lambda_i \in  (\mathbbm k G)^{\circ}$ such that $\lambda_i(g_j)=\delta_{ij}:=\left\lbrace\begin{array}{c}
	1 \text{ if } i=j,\\
	0 \text{ if } i \ne j.
	\end{array}\right.$ By the definition of $(\mathbbm k G)^{\circ}$, there exist ideals $I_i \subseteq \mathbbm k G$
	such that $\dim (\mathbbm k G/I_i) < +\infty$ and  $I_i \subseteq \Ker \lambda_i$.
	Denote $I:= I_1 \cap \dots \cap I_s$. Then $\dim (\mathbbm k G/I) < +\infty$
	and $\lambda_i$ induce $\bar \lambda_i \in  (\mathbbm k G/I)^*$ such that $\bar\lambda_i(\bar g_j)=\delta_{ij}$
	where $\bar g_i$ is the image of $g_i$ in $\mathbbm k G/I$.
	In particular, $\bar g_1, \ldots, \bar g_s$ are linearly independent.
	
	Suppose now that $\bar g_1, \ldots, \bar g_s$ are linearly independent in $\mathbbm k G/I$
	for some ideal $I\subseteq \mathbbm k G$ where $\dim(\mathbbm k G/I) < +\infty$.
	Choose any $\bar \lambda_i \in  (\mathbbm k G/I)^*$ such that $\bar\lambda_i(\bar g_j)=\delta_{ij}$.
	Denote by $\lambda_i \in (\mathbbm k G)^*$ the induced linear functions, which are zero on $I$.
	Then $\lambda_i \in (\mathbbm k G)^\circ$ and $\rho^\vee(\lambda_1\otimes (-)), \dots, \rho^\vee(\lambda_s\otimes (-))$
	form a basis in both  $\cosupp \rho$ and $\cosupp \bigl(\rho^\vee(\varkappa_{\mathbbm k G} \otimes \id_A)\bigr)$.
	Hence $\cosupp \rho = \cosupp \bigl(\rho^\vee(\varkappa_{\mathbbm k G} \otimes \id_A)\bigr)$.
\end{proof}	
\begin{corollary}\label{CorollaryNonTrivialFinDimReps}
	If $\cosupp \rho = \cosupp \bigl(\rho^\vee(\varkappa_{\mathbbm k G} \otimes \id_A)\bigr)$ and $|\supp \Gamma| \geqslant 2$,
	then there exists a nontrivial finite dimensional representation of $G$.
\end{corollary}
\begin{proof}
	It is sufficient to consider the linear representation $\varphi \colon G \to \GL(\mathbbm k G/I)$  of $G$ on $\mathbbm k G/I$ by the left multiplication:	\begin{equation}\label{EqGRepGILeftMul}\varphi(g)(t+I):= gt+I\text{ for all }g,t\in G.\end{equation}
\end{proof}

\begin{theorem}\label{TheoremInfinGrNoFinQuotientsFinDualTriv}
	Let $G$ be a finitely generated group that has no nontrivial finite quotients
	and let $\mathbbm k$ be a field. Then $(\mathbbm k G)^\circ = \mathbbm k \varepsilon \cong \mathbbm k$
	as Hopf algebras where $ \varepsilon \colon \mathbbm k G \to \mathbbm k$ is the counit (augmentation) of $ \mathbbm k G $.
\end{theorem}	
\begin{proof}
	Let $\lambda \in (\mathbbm k G)^\circ$. Denote by $I$ a two-sided ideal of $\mathbbm k G$ such that
	$\lambda(I)=0$ and $\dim(\mathbbm k G/I)<+\infty$. Again, consider the linear representation
	$\varphi \colon G \to \GL(\mathbbm k G/I)$ defined by~\eqref{EqGRepGILeftMul}.
	As it was noticed in~\cite[Proposition~2.5]{Berrick},
	by~\cite[Theorems~VII and~VIII]{AIMalcev1940}, $G$ has no nontrivial finite dimensional linear representations,
	i.e. $\varphi(g)(t+I) = t+I$ for all $g,t\in G$.
	Hence $gt-t \in I$ for all  $g,t\in G$ and $\Ker\varepsilon \subseteq I$ since the augmentation
	ideal $\Ker\varepsilon$ is the linear span of the elements $g-1$ where $g\in G$.
	Therefore, $\lambda = \alpha\varepsilon$ for some $\alpha \in \mathbbm k$.	
\end{proof}	

\begin{remark}
	For finite fields $\mathbbm k$ Theorem~\ref{TheoremInfinGrNoFinQuotientsFinDualTriv} above follows from~\cite[Exercise~2.5]{Abe}. An example of a specific simple group $G$ depending
	on an arbitrary field  $\mathbbm k$ such that $(\mathbbm k G)^\circ = \mathbbm k \varepsilon \cong \mathbbm k$  was first given in~\cite[Lemma~2.7]{BCM}. 
\end{remark}

Now we are ready to give an example of an algebra $A$ and a unital subalgebra $V\subseteq \End_{\mathbbm k}(A)$
such that the cosupport of $\psi^{\mathbf{Hopf}}_{A,V}$ is strictly less than the cosupport of $\rho^{\mathbf{Hopf}}_{A,V}$.

\begin{theorem}\label{TheoremMnElGradDualTrivial}
	Let $\mathbbm k$ be an arbitrary field and let $G$ be an infinite finitely presented group that has no nontrivial finite quotients. (E.g. $G$ is the Higman group or one of the Thompson simple groups, see~\cite{Higman1951, Higman1974}.)
	 Consider an elementary grading $\Gamma$ on the full matrix algebra $M_n(\mathbbm k)$ for some $n\in\mathbb N$ 
	such that the universal group of $\Gamma$ is isomorphic to~$G$. (Such a grading exists by~\cite[Theorem~4.3]{ASGordienko18Schnabel}.)  Denote by $\rho \colon M_n(\mathbbm k) \to M_n(\mathbbm k) \otimes \mathbbm k G$ the corresponding
	coaction and let $V := \cosupp \rho$.
	Then \begin{equation}\label{EqMnElGradDualTrivial}\cosupp \psi^{\mathbf{Hopf}}_{M_n(\mathbbm k),V} = {\mathbbm k} \id_{M_n(\mathbbm k)} \subsetneqq V = \cosupp \rho.
	\end{equation}
\end{theorem}	
\begin{remark}
		Theorem~\ref{TheoremMnElGradDualTrivial} implies that the cosupport of the dual action of a coaction may reduce and can be trivial at all.	
\end{remark}
\begin{proof}[Proof of Theorem~\ref{TheoremMnElGradDualTrivial}]
	By Theorem~\ref{TheoremUniversalHopfOfAGradingIsJustUniversalGroup}, the universal coaction $\rho^{\mathbf{Hopf}}_{M_n(\mathbbm k),V}$ coincides with $\rho$
	since the universal group of $\Gamma$ is isomorphic to~$G$ itself.
	By the Duality Theorem~\ref{TheoremUnivHopf(Co)actDuality},
	the universal action $\psi^{\mathbf{Hopf}}_{M_n(\mathbbm k),V}$ coincides with 
	$\rho^\vee(\varkappa_{\mathbbm k G} \otimes M_n(\mathbbm k))$,
	i.e. ${}_\square\mathcal H(M_n(\mathbbm k),V) = (\mathbbm k G)^\circ$.
	However, by Theorem~\ref{TheoremInfinGrNoFinQuotientsFinDualTriv},
	$(\mathbbm k G)^\circ =  \mathbbm k \varepsilon \cong \mathbbm k$,
	i.e. $\psi^{\mathbf{Hopf}}_{M_n(\mathbbm k),V}$ is trivial
	and $$\cosupp\rho^\vee(\varkappa_{\mathbbm k G} \otimes M_n(\mathbbm k)) = \cosupp \psi^{\mathbf{Hopf}}_{M_n(\mathbbm k),V} = {\mathbbm k} \id_{M_n(\mathbbm k)} \subsetneqq V = \cosupp \rho.$$
\end{proof}	

The direct application of~\cite[Theorem~4.3]{ASGordienko18Schnabel} to the Higman group would prove
the existence of such a grading on $M_n(\mathbbm k)$ only for $n\geqslant 21$. However this number can be reduced by a careful analysis of the relations:
\begin{theorem}\label{TheoremMainNonRegrFinDualTrivial} Let $G$ be the Higman group, i.e. the group with the generators $a,b,c,d$
 subject to the relations
 \begin{equation}\label{EqHigmanGroupRel}
 a^{-1} b a = b^2,\ b^{-1} c b = c^2,\ c^{-1} d c = d^2,\ d^{-1} a d = a^2
 \end{equation}
 and let $\mathbbm k$ be an arbitrary field.
 Then for every $n\geqslant 14$ there exists an elementary $G$-grading $\Gamma$ on $M_n(\mathbbm k)$
 such that:
 \begin{enumerate}
 	\item $G$ is the universal group of $\Gamma$;
 	\item $\Gamma$ is not equivalent to any grading by a finite group;
 	\item if $\rho \colon M_n(\mathbbm k) \to M_n(\mathbbm k) \otimes \mathbbm k G$ is the corresponding
 	coaction and $V := \cosupp \rho$,
 	then~\eqref{EqMnElGradDualTrivial} holds.
 \end{enumerate}	
\end{theorem}
\begin{proof}
	Define the elementary $G$-grading on $M_n(\mathbbm k)$ by
	$$e_{12} \in M_n(\mathbbm k)^{\left( a^{-1} \right)},
	\quad e_{23} \in M_n(\mathbbm k)^{\left( b \right)},
	\quad e_{34} \in M_n(\mathbbm k)^{\left( a \right)},$$
	$$e_{45} \in M_n(\mathbbm k)^{\left( b^{-1} \right)},
	\quad e_{56} \in M_n(\mathbbm k)^{\left( c \right)},
    \quad e_{67} \in M_n(\mathbbm k)^{\left( b \right)},$$
    $$e_{78} \in M_n(\mathbbm k)^{\left( c^{-1} \right)},
	\quad e_{89} \in M_n(\mathbbm k)^{\left( d \right)},
    \quad e_{9,10} \in M_n(\mathbbm k)^{\left( c \right)},$$
    $$e_{10,11} \in M_n(\mathbbm k)^{\left( d^{-1} \right)},
	\quad e_{11,12} \in M_n(\mathbbm k)^{\left( a \right)},
    \quad e_{12,13} \in M_n(\mathbbm k)^{\left( d \right)},$$
    $$e_{13,14} \in M_n(\mathbbm k)^{\left( a^{-1} \right)},
    \quad e_{r,r+1} \in M_n(\mathbbm k)^{\left( 1 \right)}\text{ for }14\leqslant r \leqslant n-1.$$
    
    Then 
	$$e_{21}, e_{14,13} \in M_n(\mathbbm k)^{\left( a \right)},\quad
e_{54} \in M_n(\mathbbm k)^{\left( b \right)},\quad
e_{87} \in M_n(\mathbbm k)^{\left( c \right)},\quad e_{11,10} \in M_n(\mathbbm k)^{\left( d \right)}.$$
    
    Note that the relations~\eqref{EqHigmanGroupRel} imply that
    $$e_{15}=e_{12}e_{23}e_{34}e_{45} \in M_n(\mathbbm k)^{\left( b \right)},\quad
    e_{48}=e_{45}e_{56}e_{67}e_{78} \in M_n(\mathbbm k)^{\left( c \right)},$$
    $$e_{7,11}=e_{78}e_{89}e_{9,10}e_{10,11} \in M_n(\mathbbm k)^{\left( d \right)},\quad
    e_{10,14}=e_{10,11}e_{11,12}e_{12,13}e_{13,14} \in M_n(\mathbbm k)^{\left( a \right)}.$$
    
  Suppose $\Gamma$ is regraded by some group $\tilde G$. Denote by $\tilde \Gamma$ the corresponding grading
  and by $\tilde a, \tilde b, \tilde c, \tilde d$ the elements of $\tilde G$ such that $e_{21}, e_{34}, e_{10,14}, e_{11,12}, e_{14,13} \in M_n(\mathbbm k)^{\left( \tilde a \right)}$,
  $e_{15}, e_{23}, e_{54}, e_{67} \in M_n(\mathbbm k)^{\left( \tilde b \right)}$,
  $e_{48}, e_{56}, e_{87}, e_{9,10} \in M_n(\mathbbm k)^{\left( \tilde c \right)}$,
  $e_{7,11}, e_{89}, e_{11,10}, e_{12,13} \in M_n(\mathbbm k)^{\left( \tilde d \right)}$.
  Then the multiplication rule for matrix units implies that 
  	$$e_{12}, e_{13,14} \in M_n(\mathbbm k)^{\left( \tilde a^{-1} \right)},\qquad
	e_{45} \in M_n(\mathbbm k)^{\left( \tilde b^{-1} \right)},$$
     $$e_{78} \in M_n(\mathbbm k)^{\left( \tilde c^{-1} \right)},\quad
    e_{10,11} \in M_n(\mathbbm k)^{\left( \tilde d^{-1} \right)}$$
 and the elements of $\supp \tilde\Gamma$ belong to the subgroup of $\tilde G$ generated by $\tilde a, \tilde b, \tilde c, \tilde d$. Moreover, applying the equalities $$e_{15}=e_{12}e_{23}e_{34}e_{45}, \quad e_{48}=e_{45}e_{56}e_{67}e_{78},$$
 $$e_{7,11}=e_{78}e_{89}e_{9,10}e_{10,11}, \quad  e_{10,14}=e_{10,11}e_{11,12}e_{12,13}e_{13,14}$$
 once again, we obtain that 
 $$ \tilde a^{-1} \tilde b\tilde a = \tilde b^2,\ \tilde b^{-1} \tilde c \tilde b =\tilde  c^2,\ \tilde c^{-1} \tilde d \tilde c = \tilde d^2,\ \tilde d^{-1} \tilde a \tilde d = \tilde a^2.$$

 Therefore, there exists a unique group homomorphism $G \to \tilde G$ that maps the elements of $\supp \Gamma$
 to the corresponding elements of $\supp \tilde\Gamma$. Thus $G \cong G_\Gamma$.
 
 Since the Higman group $G$ has no nontrivial finite quotients~\cite{Higman1951}, $\Gamma$ cannot be regraded by any finite group.
  Property~(3) follows from Theorem~\ref{TheoremMnElGradDualTrivial}.
\end{proof}	

\section*{Acknowledgments}

We are grateful to Anton Klyachko for suggesting that we consider the Higman group. In addition, we would like to thank the referee for careful reading of our manuscript, pointing out misprints and providing useful suggestions.

\end{document}